\documentclass{amsart}[12pt]

\usepackage{amsfonts,amssymb,amsmath,amscd,amsthm, mathrsfs}
\usepackage{graphicx}
\usepackage{epstopdf}

\newtheorem{thm}{Theorem}[section]
\newtheorem{dfi}{Definition}[section]
\newtheorem{prop}{Proposition}[section]
\newtheorem{lm}{Lemma}[section]
\newtheorem{cor}{Corollary}[section]
\newtheorem{rem}{Remark}[section]

\newcommand{\mb}[1]{\mathbb{#1}}
\newcommand{\mf}[1]{\mathfrak{#1}}
\newcommand{\mc}[1]{\mathcal{#1}}
\newcommand{\rr}{\mathbb{R}}
\newcommand{\zz}{\mathbb{Z}}
\newcommand{\cc}{\mathbb{C}}

\newcommand{\al}{\alpha}

\newcommand{\dl}{\delta}

\newcommand{\ep}{\varepsilon}
\newcommand{\Gm}{\Gamma}
\newcommand{\gm}{\gamma}

\newcommand{\lmd}{\lambda}
\newcommand{\om}{\omega}
\newcommand{\Om}{\Omega}

\newcommand{\nb}{\nabla}
\newcommand{\tht}{\theta}

\newcommand{\cx}{\mathcal{X}}

\newcommand{\la}{{\lambda}^{-\frac{2}{2 + \al}}}
\newcommand{\power}{\frac{2}{2 + \al}}
\newcommand{\powerr}{\frac{\al}{2 + \al}}

\newcommand{\cl}{c^{\lambda}}

\newcommand{\yl}{y_{\lambda}}

\newcommand{\xl}{x_{\lambda}}
\newcommand{\yb}{\bar{y}}
\newcommand{\Mb}{\bar{M}}
\newcommand{\lt}{\lambda t}

\newcommand{\mbb}{\bar{m}}
\newcommand{\ro}{\rho}
\newcommand{\Gmr}{\Gm(\xb; \rho)}

\newcommand{\rob}{\bar{\rho}}
\newcommand{\ron}{\rho_n}
\newcommand{\robn}{\bar{\rho}_n}

\newcommand{\dn}{\delta_n}
\newcommand{\yln}{y_{\lambda_n}}
\newcommand{\zln}{z_{\lambda_n}}
\newcommand{\xln}{\xi_{\lambda_n}}

\newcommand{\Vb}{\bar{V}}
\newcommand{\ttn}{\theta_n}
\newcommand{\tna}{t_n^-}
\newcommand{\tnb}{t_n^+}

\newcommand{\h}{H^1}
\newcommand{\aet}{A_{\frac{T}{2}}}

\newcommand{\xb}{\bar{x}}
\newcommand{\cb}{\bar{c}}
\newcommand{\lb}{\bar{L}}
\newcommand{\ab}{\bar{A}}
\newcommand{\vb}{\bar{V}}

\newcommand{\ey}{\frac{1}{2}}

\newcommand{\gts}{\Gm^*_T(\tau)}
\newcommand{\gtt}{\Gm_T(\tau) }

\newcommand{\gmp}{\Gm^+_{\dl}(y)}
\newcommand{\gmm}{\Gm^-_{\dl}(y)}

\newcommand{\aly}{\frac{2}{2 + \al}}
\newcommand{\ale}{\frac{2 \al}{2 + \al}}

\newcommand{\vlm}{V_{\lmd}}
\newcommand{\llm}{L_{\lmd}}
\newcommand{\alm}{A^{\lmd}}
\newcommand{\aln}{A^{\lmd_n}}

\newcommand{\lmn}{\lmd_n}

\newcommand{\xr}{x_{\rho}}
\newcommand{\tr}{t_{\rho}}

\newcommand{\tb}{\bar{T}}

\newcommand{\xt}{\tilde{x}}

\newcommand{\qt}{\tilde{q}}

\begin{document}

\title[periodic solutions topological constraints]{Periodic Solutions of the Planar N-Center Problem with topological constraints}
\author{Guowei Yu}

\address{Department of Mathematics, University of Toronto
}

\begin{abstract} 
In the planar $N$-center problem, for a non-trivial free homotopy class of the configuration space satisfying certain mild condition, we show that there is at least one collision free $T$-periodic solution for any positive $T.$ We use the direct method of calculus of variations and the main difficulty is to show that minimizers under certain topological constraints are free of collision. 
\end{abstract}

\maketitle

\section{Introduction} \label{sec intro}

In the study of the classic $N$-body problem, or the general singular Lagrangian systems, one of the oldest ideas in finding periodic solution, at least goes back to Poincar\'e, is by looking for minimizers of the action functional in certain admissible classes of loops. 

Besides the usual coercive condition, comparing with regular systems, the extra difficulty is to show that the desired minimizer is free of singularity. Otherwise we get something called generalized solutions, see \cite{BR89}, \cite{AC93}. 

If the singularities are caused by the \emph{strong force} potentials (see Remark \ref{potential} for the precise definition), then the extra difficulty is gone (again already known to Poincar\'e). As in this case the action functional along any loop with singularity must be infinite. All kinds of periodic solutions can be found using variational methods, for example see see \cite{Go75}, \cite{AC93}, \cite{Mo98} and \cite{CGMS02}.

However when the singularities are caused by the \emph{weak force} potentials, including the Newtonian potential, then the action functional along a loop with singularities may still be finite. In fact by the results of Gordon \cite{Go77} and Venturelli \cite{Ve01}, we know in certain cases there are minimizers with singularities. 

Started with the (re)-discovery of the Hip-Hop solution \cite{CV01} in the Newtonian Four-body problem and the Figure-Eight solution \cite{CM00} in the Newtonian three-body problem. We have learned that one of the method to overcome the above problem is to impose certain symmetric constraints on the admissible class of loops. For the details please see \cite{FT04}, \cite{BFT08},  \cite{C02}, \cite{Ch03}, \cite{Ch08}, \cite{FGN11}, \cite{Sh14} and the references within.

On the other hand very few results are available when topological constraints are involved and it is still a difficult task to determine whether a minimizer is free of collision in this case. 



In order to clarify and overcome the difficulties related to the topological constraints, in this paper we propose to study the simplified model, namely the planar $N$-center problem, with $N \ge 2$. Since all the natural symmetries of the $N$-body problem do not exist anymore, it will be a perfect test ground for various techniques that have been developed by many people in the past twenty years. The main result of our paper is some simple criteria on the topological constraints that will ensure the existence of collision free minimizers (in the $N$-center problem the singularities are caused by collisions between the test particle and a center). 

In forth coming papers we will show that some of the results proved in this paper can be generalized to the $N$-body problem.

In the planar N-center problem, a test particle is moving in the plane under the gravitational field of $N$ fixed point masses (\emph{centers}). We use $\{ m_j > 0: j= 1, \dots, N \}$ and $\mc{C}:= \{ c_j \in \cc: j = 1, \dots, N \}$ ( $c_j \ne c_k ,$ if $j \ne k$) to denote the masses and positions of the $N$ centers correspondingly.

If $x: \rr \to \cx$, where $\cx: =\cc \setminus \mc{C},$ is the position function of the test particle, then $x(t)$ should satisfy the following second order differential equation:

\begin{equation}
 \label{ncenter} \ddot x(t) = \nb V(x(t))=  - \sum_{j = 1}^{N} \frac{m_j}{|x(t) - c_j|^{\al+2}}(x(t) - c_j),
\end{equation}

for $\al >0$ and
$$ V(x(t)) : = \sum_{j=1}^{N} \frac{m_j}{\al |x(t) - c_j|^{\al}}, $$
is the negative potential at $x(t)$.

\begin{rem} \label{potential}
        In general, $\al \ge 2$ correspond to the \emph{strong force} potentials, $0<\al <2$ correspond to the \emph{weak force} potentials and in particular $\al=1$ corresponds the Newtonian potential. In this paper we will not discuss $\al \in (0,1)$ and to distinguish with the Newtonian case, by weak force potential, we will only mean $\al \in (1, 2). $
\end{rem}

Equation \eqref{ncenter} has a natural variational formulation.  It is the Euler-Lagrange equation of the action functional
$$ A(x) = \int L(x(t), \dot{x}(t)) \, dt, $$
where the Lagrangian $L$ has the form
$$ L(x, \dot{x}): = \frac{1}{2} |\dot{x}|^2 + V(x). $$

For any $T_1 < T_2$, $T>0$ and $U \subset \cc$, we let $H^1([T_1, T_2], U)$ be the space of all Sobolev functions mapping $[T_1, T_2]$ to $U$ and $H^1(S_T, U)$ be the space of all $T$-periodic Sobolev loops contained in $U$, where $S_T: = [-\frac{T}{2}, \frac{T}{2}] / \{ -\frac{T}{2}, \frac{T}{2} \}.$ 

Given an $x \in H^1([T_1, T_2], U)$, define its action as
$$ A([T_1, T_2]; x) = \int_{T_1}^{T_2} L (x(t), \dot{x}(t)) \,dt. $$
In particular if $T= -T_1 = T_2 >0$, then we set
$$ A_T(x) := A([-T, T]; x). $$


For any nontrivial free homotopy class $\tau \in \pi_1(\cx) \setminus \{0\}$ and $T>0$, let 
 $$ \gts : = \{ x \in H^1(S_T, \cx): [x] = \tau\} $$
and $\gtt$ be the weak closure of $\gts$ in $H^1(S_T, \cc)$.

Our goal is to find some simple criteria on $\tau$, such that the following infimimum 
$$ c_{T}(\tau) = \inf \{ A_{T/2}(x): x \in \gtt \}$$
can be achieved at some collision free loop from $\gts$.

First we will introduce some notations. Given an $x \in H^1(S_T, \cx)$, we say it is a \textbf{generic} $T$-periodic loop, if $x$ is a smooth immersion in \textbf{general position} (for a definition, see page 82, \cite{Hi76}). Such a loop only contains transverse self-intersections and we say it has \textbf{excess self-intersection} if it can be homotopic to another generic loop with fewer self-intersection (see \cite{HS85}).  

If $I \subset S_T$ is a sub-interval with $x$ identifies the end points of $I$, we say $x|_I$ is a sub-loop of $x$. Such a sub-loop is called \emph{innermost}, if it is also a \emph{Jordan curve}. As a result each innermost sub-loop separates the plane into two disjoint regions, one bounded and the other not. We say a center is \emph{enclosed} by an innermost sub-loop if it is contained in the bounded region. 

\begin{dfi}
	Given a $\tau \in \pi_1(\cx)$, choose an arbitrary generic loop $x \in \gts$ without any excess self-intersection. We say $\tau$ is \textbf{admissible}, if any innermost sub-loop of $x$ encloses at least two different centers.
\end{dfi}

\begin{figure}
	\centering
	\includegraphics[scale=0.7]{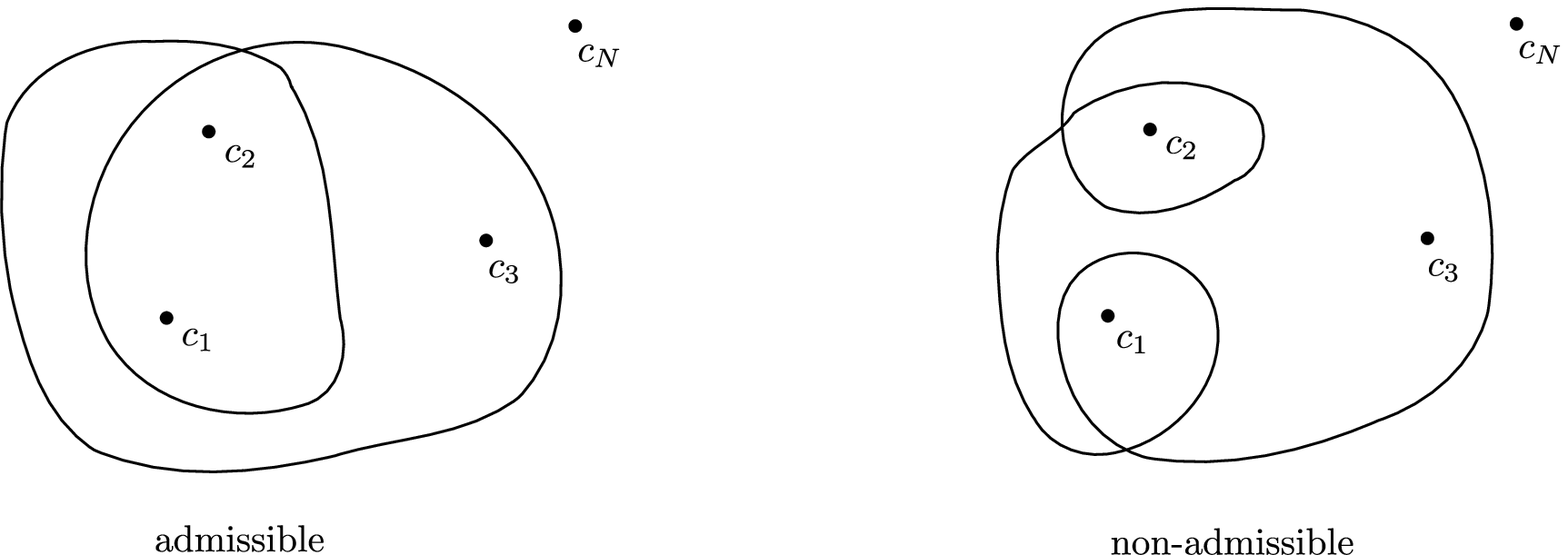}
	\caption{}
\end{figure}

Now we are ready to state the main results of our paper.

\begin{thm} \label{thm 1}
  When  $\al \in (1, 2)$, for any admissible free homotopy class $\tau \in \pi_1(\cx)$ and $T>0$, there is a $q \in \gts$ with  $A_{T/2}(q) = c_{T}(\tau)$ and it is a $T$-periodic solution of equation \eqref{ncenter}. 
\end{thm}

For the Newtonian potential, we get a slightly weaker result. To state our result, we consider the Hurewicz homomorphism: $ \mf{h}: \pi_1(\cx) \to H_1(\cx) \cong \zz^N$, which in our case is just the canonical abelianization map. Given an free homotopy class $\tau$, choose a generic $x \in \gts$, we can define $\mf{h}(\tau)=(\mf{h}(\tau)_j)_{j=1}^N \in \zz^N$ as following
 $$ \mf{h}(\tau)_j = \text{ind}(x, c_j)= \frac{1}{2 \pi i}\int_{-T/2}^{T/2} \frac{dx}{x -c_j}, \quad \forall j =1, \dots, N, $$
where $\text{ind}(x, c_j)$ is the index of $x$ with respect to $c_j.$ It is independent of the choice of $x$, as long as $[x]=\tau.$

\begin{thm} \label{thm 2}
  When $\al =1$, for any admissible free homotopy class $\tau \in \pi_1(\cx)$ and $T>0$, there is a $q \in \gtt$ with $A_{T/2}(q) = c_{T}(\tau)$. Furthermore one of the following must be true:
 \begin{enumerate}
 	\item $q \in \gts$ and it satisfies equation \eqref{ncenter};
 	\item there are two centers $c_{k_1}, c_{k_2} \in \mc{C}$ (possibly $c_{k_1}= c_{k_2}$), a $\tb >0$ with $T / 2\tb \in \zz^+$ and $q$ satisfies
 	\begin{enumerate}
 		\item $ q(0) = c_{k_1}, q(\tb) = c_{k_2};$
 		\item $q(t) = q(-t) \in \cx$ and $q(t)$ satisfies equation,  \eqref{ncenter}, for any $t \in (0,\tb);$
 		\item $q(t + 2 \tb) = q(t),$ for any $t \in \rr.$

 	\end{enumerate}
 	
 \end{enumerate}
 The second case can happen only when $\tau$ satisfies 
 $$ \mf{h}(\tau)_j =0 , \quad \forall j \in \{1, \dots, N \} \setminus \{k_1, k_2 \}. $$
 \end{thm}

\begin{figure}
	\centering
	\includegraphics[scale=0.7]{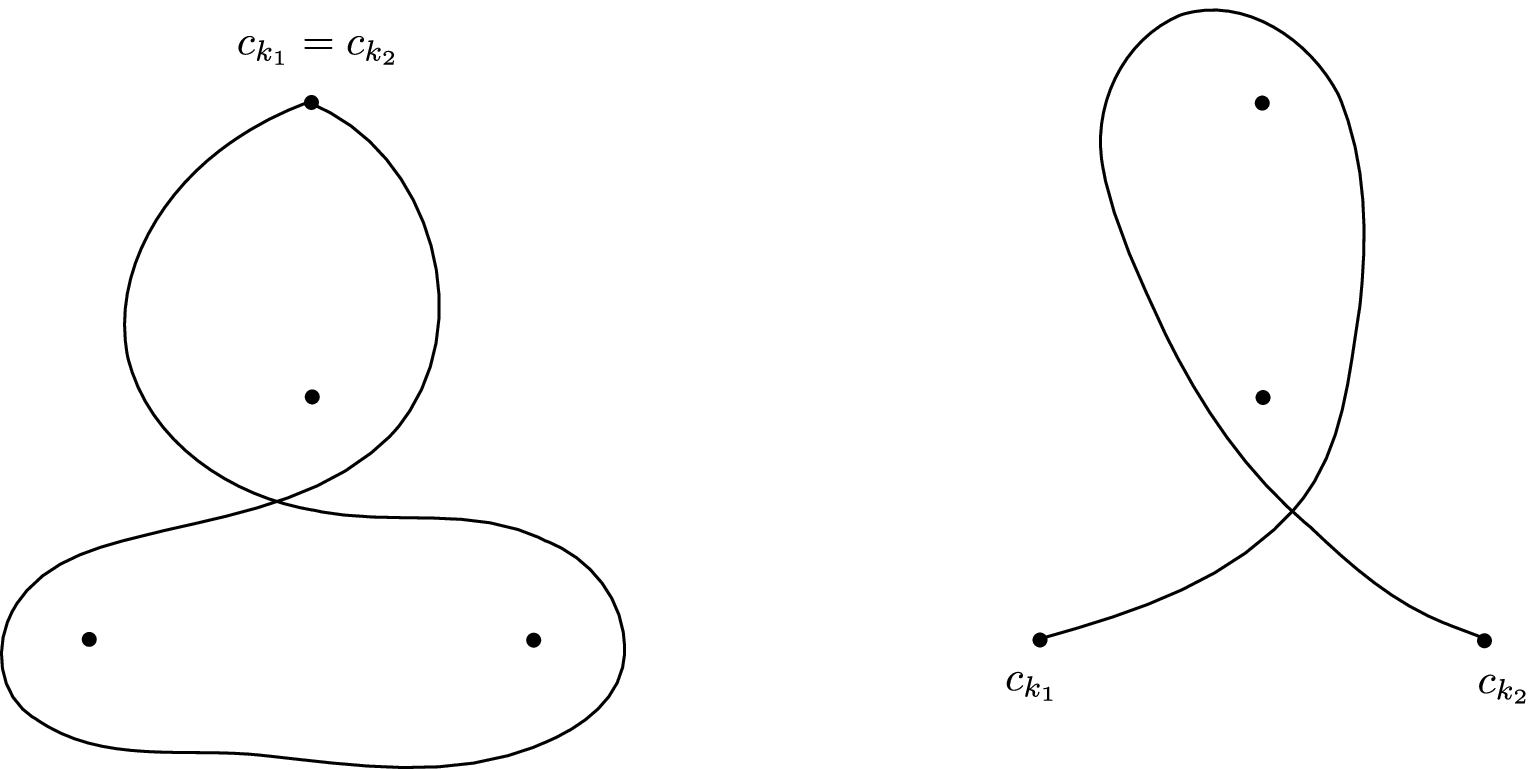}
	\caption{}
\end{figure}

\begin{rem}
	In the above theorem, if a solution satisfies the conditions of the second case, we say it is a \textbf{collision-reflection} solution between $c_{k_1}$ and $c_{k_2}$. See Figure $2$, for some illuminating pictures of such solutions.  
	
\end{rem}

An immediate corollary of Theorem \ref{thm 2} is that 

\begin{cor}
	For any $ N \ge 3, T>0$ and $\al =1$, if an admissible free homotopy class $\tau \in \pi_1(\cx)$ satisfies $\mf{h}(\tau)_j \ne 0$ for at least three different $j$'s from $\{1, \dots, N \}$, then there is a $q \in \gts$ with $A_{T/2}(q) = c_{T}(\tau)$ and it is a $T$-periodic solution of equation \eqref{ncenter}.
\end{cor}

Our work is particularly inspired by the paper of Soave and Terracini \cite{ST12}, where they also studied periodic solutions of $N$-center problem with topological constraints using variational method. The main difference is they studied the fixed energy case, while we are working on the fixed time case. 

Comparing with the results in \cite{ST12}, the variational method seems to work better in the fixed time case. In \cite{ST12} the results were proved under the assumption that the energy level is close enough to zero, while in our case, the results hold for any $T>0$. More importantly, when one tries to generalize the results to the \emph{rotating $N$-center} problem, where instead of fixing the positive masses at certain locations, one assumes they are rotating with respect to the origin at a uniform angular velocity. For the fixed energy case, this can only been done for angular velocity close enough to zero, as shown in \cite{So14}; while for the fixed time case, in \cite{Y15d} we are able to generalize these results for an arbitrary angular velocity. 

Recall that in the restricted $N+1$-body problem, a special case is the $N$ positive masses  form a relative periodic solution of the $N$-body problem, which makes it become a rotating $N$-center problem. However the uniform angular velocity is defined by the central configuration of the relative periodic solution and it is generally not close to zero, unless all the masses are close enough to infinity. As a result while the fixed energy case can not be applied to the restricted $N+1$-body problem, our results of the fixed time case does apply, see \cite{Y15d}. Furthermore combining some perturbation argument, we believe it can also be used to prove the existence of (relative) periodic solutions of $N+1$-body problem in certain braid classes, for a definition of the braid classes see \cite{Mo98}, under the assumption that mass of one of the bodies is small enough. This will discussed in a forthcoming paper.  

As we mentioned the main difficulty is to show that the minimizers are free of collision. There are essential two  approaches. One is the so called \emph{level estimate} and the other is \emph{local deformation}, for details see \cite{Y15c} and the references within. Here we will use the second approach. 

The second approach usually contains two steps: First by the \emph{blow-up} technique introduced by Terracini, see \cite{Ve02}, \cite{FT04}, an \emph{isolated collision solution} (see Definition \ref{coll solution}) will be translated to a \emph{parabolic collision-ejection solution} (see Definition \ref{coej}) of the Kepler-type problem in our case (a homothetic-parabolic solution in the $N$-body problem); Second we will show the parabolic collision-ejection solution is not a minimizer in certain admissible class of curves. 

Different ideas can be used in the second step depending on the nature of the problem. Here we will consider an obstacle minimizing problem, which has been studied by Terracini and her coauthors in various papers, \cite{TV07}, \cite{ST12} and \cite{BTV13}.






We state the result obtained in the second step as a separate theorem in the following, as we believe it will be useful in different variational approaches when topological constraints are involved, for example see \cite{Y15b}, \cite{Y15c}. 

Consider the planar Kepler-type problem or one-center problem, where the motion of the test particle satisfies the following equation
\begin{equation}
 \label{Kepler} \ddot{x}(t)= \nb \bar{V}(x(t))= -\frac{m_1 x(t)}{|x(t)|^{\al+2}},
\end{equation}

where $\bar{V}(x(t)):= \frac{m_1}{\al|x(t)|^{\al}}$ is the negative potential at $x(t)$. The corresponding Lagrange and action functional are defined as following 
$$ \lb (x, \dot{x}) := \frac{1}{2} |\dot{x}|^2 + \vb(x), $$
$$ \ab ([T_1, T_2]; x) := \int_{T_1}^{T_2} \lb (x(t), \dot{x}(t)) \,dt.$$
 
As before for any $T> 0$, we set
$$ \ab_T(x) := \ab([-T, T]; x). $$

\begin{dfi} \label{coej}
Given an arbitrary pair of angles $\phi_-, \phi_+$, we define a parabolic \emph{collision-ejection} solution as
\begin{equation} \label{xbar}
 \xb(t) =
\begin{cases}
  (\mu |t|)^{\frac{2}{\al+2}} e^{i \phi_+}  & \text{ if } t \in [0, +\infty), \\
  (\mu |t|)^{\frac{2}{\al+2}} e^{i \phi_-} & \text{ if } t \in (-\infty, 0],
\end{cases}
\end{equation}
where
\begin{equation} \label{mu}
\mu = (\al +2) (\frac{m_1}{2 \al})^{\frac{1}{2}}.
\end{equation}
\end{dfi}

A simple calculation shows that a parabolic collision-ejection solution is a solution of the Kepler-type problem with zero energy, except at $t=0$.

For any $T>0$, define the following class of curves:
$$ \Gm^*_T(\xb):= \{ x \in \h([-T, T], \cc \setminus \{0 \}): x \text{ satisfies the following conditions } \}. $$
\begin{enumerate}
 \item $x(\pm T) = \xb( \pm T);$
 \item $ \text{Arg}(x(\pm T)) = \phi_{\pm}.$
\end{enumerate}

Let $\Gm_T(\xb)$ be the weak closure of $\Gm^*_T(\xb)$ in $\h([-T, T], \cc)$ and
$$ \cb_T(\xb):= \inf \{ \ab_T(x): x \in \Gm_T(\xb) \}. $$

\begin{thm}
 \label{thm 3} For any $T>0$ and $|\phi_+ - \phi_-| \le 2 \pi$,  there is a $\gm \in \Gm_T(\xb)$, with $\ab_T(\gm) = \cb_T(\xb)$.
 \begin{enumerate}
  \item If $\al \in (1,2),$ then $ \gm \in \Gm^*_T(\xb)$ and $\ab_T(\gm) < \ab_T(\xb)$;
  \item If $\al =1$ and $|\phi_+ - \phi_-| < 2\pi$, then $ \gm \in \Gm^*_T(\xb)$ and $\ab_T(\gm) < \ab_T(\xb)$. 
 \end{enumerate}
 Furthermore $\gm$ is a solution of the Kepler-type problem, if $\phi_+ \ne \phi_-.$

\end{thm}
The proof of the above result will be given at the last section. They main idea of the proof comes from \cite{ST12}. For the case of $\al =1$, a different proof of the above result can also found in \cite{FGN11}.  
\begin{rem}
By a classic result of Gordon in \cite{Go77}, we know that when $\al =1$ and $|\phi_+- \phi_-|= 2\pi $ the above result does not hold, so this is the best result we can get in the Newtonian case.
\end{rem}



\section {Asymptotic analysis} \label{sec asym}
The asymptotic behavior of a solution as it approaching a collision, which will be established in the section, is the base of any local deformation result. Most of the results in this section are not new, similar results in various settings have been obtained before, see \cite{FT04} and the references there. We include them here for the seek of completeness and also because the direct references of these results in our particular setting seems hard to locate.

\begin{dfi} \label{coll solution}
 Given a $\dl>0$ and time $t_0$, we say $y: [t_0-\dl, t_0+\dl] \to \cc$ is an \textbf{isolated collision solution}, if it satisfies the following conditions
 \begin{enumerate}
 	\item $y(t_0) \in \mc{C};$
 	\item $y(t) \in \cx$ satisfies equation \eqref{ncenter}, if $t \ne t_0;$
 	\item The energy constant is preserved through the collision, i.e., there is a $h$, such that,  if $t \ne t_0,$
 	$$ \ey |\dot{y}(t)|^2 -V(y(t)) = h$$
 \end{enumerate}
 Furthermore we say $y$ is a \textbf{simple isolated collision solution}, if it does not have any transverse self-intersection.
 
\end{dfi}

Let $y(t)$ be an isolated collision defined as above, to simplify notation, we will only state and prove our results with $t_0 =0$ and $y(t_0) =c_1$, while the general cases are exactly the same. Furthermore for the rest of this paper, we assume $c_1$ is always located at the origin: $c_1 =0$.  

The main results of this section are Proposition \ref{angular limit} and \ref{blowup limit}. In this section we do not require the isolated collision solution to a be minimizer. 

Define the \textbf{moment of inertia} and \textbf{angular momentum} of $y(t)$ with respect to the center $c_1=0$ correspondingly as
$$I(t):= I(y(t)):= |y(t)|^2;$$
$$ J(t):= J(y(t)):= y(t) \times \dot{y}(t).$$
Set $y(t) =   \rho(t) e^{i\tht(t)}= u(t) + i v(t),$ with $\ro(t), \tht(t), u(t), v(t) \in \rr$, then
$$ J(t)= \rho^2(t)\dot{\tht}(t) = u(t)\dot{v}(t)- v(t)\dot{u}(t) .$$

Since $y|_{[-\dl, \dl]}$ only collides with $c_1$, the negative potential with respect to the other $N-1$ centers is a smooth function:
$$ V^*(y) := V(y) - \frac{m_1}{\al |y|^{\al}} = \sum_{j =2}^{N} \frac{m_j}{\al |y - c_j|^{\al}}.$$

Now we will introduce two technical lemmas about $I(t)$ and $J(t)$ that will be useful throughout the entire section.

The first lemma the well known \textbf{Lagrange-Jacobi identity}.
\begin{lm} \label{LaJa}
There are continuous functions $B_1(t), B_2(t) \in C^0( [-\dl, \dl], \cc)$, such that the following identities hold for all $ t \in [-\dl, \dl]\setminus \{ 0 \}$
\begin{align*}
 \ddot{I}(t) & = (4 - 2\al)\frac{|\dot{y}(t)|^2}{2} + B_1(t) = (4-2 \al)[h+V(y(t))] + B_1(t) \\
   & = (4-2\al) \frac{m_1}{\al|y(t)|^{\al}} + B_2(t)
\end{align*}
\end{lm}

\begin{proof}
By a straight forward computation 
$$ \ddot{I}= 2 | \dot{y}|^2 - \frac{2m_1}{|y|^{\al}} + 2 \langle \nabla V^*(y), y \rangle. $$
Using the energy identity
$$ \ey |\dot{y}|^2 - V(y) = \ey |\dot{y}|^2 - \frac{m_1}{\al |y|^{\al}}- V^*(y) = h, $$
we get 
\begin{align*}
\ddot{I} &= (4 -2 \al)\frac{|\dot{y}|^2}{2}- 2 \al (h + V^*(y))+ 2 \langle \nabla V^*(y), y \rangle \\
&:= (4 - 2 \al)\frac{|\dot{y}|^2}{2}+ B_1.
\end{align*}
It is not hard to see $B_1(t) \in C^0([-\dl, \dl], \cc). $ The rest of the lemma can be shown similarly using the energy identity. 

\end{proof}
\begin{rem} \label{dotB}
From the above proof, we can see $\dot{B}_1(t)$ is well-defined for any $t \ne 0$ and there is a positive constant $C$ such that 
$$ |\dot{B}_1(t)| \le C |\dot{y}(t)|, \quad \forall t \in [-\dl, \dl]\setminus \{0 \}.$$
\end{rem}

The second lemma is about the angular momentum $J(t)$.

\begin{lm} \label{J0}
 $$\lim_{t \to 0} J^2(t) = 0.$$
\end{lm}

\begin{proof}
 For any $t \in [-\dl, \dl] \setminus \{0 \},$
 \begin{align*}
  J^2 & = (u \dot{v} - v\dot{u})^2 \le (u^2 + v^2 ) ( \dot{u}^2 + \dot{v}^2) \\
      & = 2|y|^2  \frac{|\dot{y}|^2}{2} = 2|y|^2 [ h+ V(y)] \\
      & = 2 \big[ \frac{m_1}{\al} |y|^{2 - \al} + (h+ V_1(y)) |y|^2 \big] \\
      & := C|y|^{2 - \al} + B(t)|y|^2,
 \end{align*}
where $C$ is a positive constant and $B(t) \in C^0([-\dl, \dl], \cc)$. The desired result follows from the facts that $2-\al>0$ and $\lim_{t \to 0} |y(t)| = 0$
\end{proof}

With the above two lemmas, we can get the following asymptotic estimates of $I(t)$ when $y(t)$ is approaching to a collision.
\begin{prop}
 \label{asmp}
 If $y|_{[-\dl, \dl]}$ is an isolated collision solution with $y(0)=0$ and $\mu$ defined as in \eqref{mu}, then 
 
 \begin{enumerate}
  \item for $t>0$,  $$ I(t) \sim (\mu t)^{\frac{4}{2 + \al}}; \quad \dot{I}(t) \sim \frac{4}{2 + \al} \mu (\mu t)^{\frac{2 - \al}{2 + \al}}; \quad  \ddot{I}(t) \sim 4 \frac{2 - \al}{(2 + \al)^2} \mu^2 ( \mu t)^{\frac{-2\al}{2 + \al}}.$$
 \item for $t<0$,
  $$ I(t) \sim (-\mu t)^{\frac{4}{2 + \al}}; \quad \dot{I}(t) \sim - \frac{4}{2 + \al} \mu (-\mu t)^{\frac{2 - \al}{2 + \al}}; \quad  \ddot{I}(t) \sim 4 \frac{2 - \al}{(2 + \al)^2} \mu^2 ( -\mu t)^{\frac{-2\al}{2 + \al}}.$$
  
 \end{enumerate}
 Furthermore
$$ \frac{1}{2} | \dot{y}(t)|^2 \sim V(y(t)) \sim \frac{1}{4 - 2 \al} \ddot{I}(t) \sim \frac{2}{(2 + \al)^2} \mu^2 ( \mu |t|)^{\frac{-2\al}{2 + \al}}.$$

\end{prop}

This type of asymptotic estimates has been proved by many authors in different circumstances. We refer the readers to \cite{FT04} and the references there. We will only show the proof for $t \in (0, \dl]$, while the other is similar. By quoting $\bf(6.2)$, page $323$ at \cite{FT04}, the above proposition follows immediately from the next lemma.

\begin{lm}
For $t>0,$ the moment of inertia $I(t)$ satisfies the following results:
\begin{enumerate}
\item $\lim_{t \to 0^+} I(t) =0;$
\item $\forall t \in (0, \dl]$, $I(t) \ge 0$ and $\dot{I}(t) \ge 0;$
\item There are constants $a$, $b = \frac{4}{2-\al} > 2$ and $c$ such that $ \dot{I}^2 +a^2 \le b I \ddot{I} + c I;$
\item There is a constant $d>0,$ such that for any $ t \in (0, \dl],$ $\ddot{I}(t) I^{\frac{b-2}{b}}(t) \ge d;$
\item There is a constant $e>0$, such that for any $  t \in (0, \dl] ,$ $|\frac{d^3 I(t)}{d t^3}| < e \ddot{I}^{\gm}(t),$ with $\gm = \frac{2b-3}{b-2} = \frac{3}{2} + \frac{1}{\al}.$
\end{enumerate}

\end{lm}

\begin{proof}
 ($1$). Obvious.

 ($2$). The first inequality is obvious while the second follows from Lagrange-Jacobi identity.

 ($3$). If $y = u+iv$ with $u, v \in \rr$, then simple calculation shows that
 \begin{equation}
  \label{I1} (\frac{\dot{I}}{2})^2 + J^2 = (u^2 + v^2) (\dot{u}^2 + \dot{v}^2) = I |\dot{y}|^2. \end{equation}

 By the Lagrange-Jacobi identity,
 $$ \ddot{I}(t) = (2 -\al) |\dot{y}(t)|^2 + B_1(t). $$
 Therefore
 \begin{equation} \label{I7} |\dot{y}(t)|^2 = \frac{\ddot{I}(t) -B_1(t)}{2 -\al} \le \frac{\ddot{I}(t)}{2 - \al} +C_1. \end{equation}

Plug \eqref{I7} into \eqref{I1}, we get
$$ \dot{I}^2 +4 J^2 \le \frac{4}{2 -\al} I \ddot{I} + 4C_1 I. $$
Then for $b = \frac{4}{2 -\al}$ and $c = 4C_1$, we get
$$\dot{I}^2 \le  \dot{I}^2 +4 J^2 \le b I \ddot{I} + c I. $$

($4$). Again by the Lagrange-Jacobi identity, for any $t \in [-\dl, \dl] \setminus \{ 0 \},$

\begin{equation}
 \label{I2}\ddot{I} \ge (4 -2 \al)\frac{m_1}{ \al|y|^{\al}} - C_2.
\end{equation}

Then for $\dl$ small enough,
$$ \ddot{I} I^{\frac{b-2}{b}} = \ddot{I} |y|^{\al} \ge \frac{(4-2\al)m_1}{\al} - C_2 |y|^{\al} \ge d > 0.$$

($5$). For $t \ne 0$, after taking the derivatives of both sides of
$$ \ddot{I}(t) = (4-2\al)[h +V(y(t))]+ B_1(t),$$
we get
\begin{equation} \label{I3}
\frac{d^3 I}{d t^3}(t) = (4 -2\al) \frac{d V(y(t))}{dt} + \dot{B_1}(t) =  (4 -2\al) \langle \nabla V(y(t)), \dot{y}(t) \rangle + \dot{B_1}(t). \end{equation}
Notice that
\begin{equation}
 \label{I4} | \nabla V(y)| \le \frac{m_1}{|y|^{\al+1}} + \sum_{j =2}^{N} \frac{m_j}{|y - c_j|^{\al +1}} \le m_1 |y|^{-(\al+1)} +C_3,
\end{equation}
and by \eqref{I7}, \eqref{I2},
\begin{equation}
 \label{I6} |\dot{y}| \le C_{4} (\ddot{I})^{\frac{1}{2}} +C_{5}
\end{equation}
\begin{equation}
\label{I5} |y|^{-(\al+1)} \le  C_6 \ddot{I}^{\frac{\al +1}{\al}} + C_7.
\end{equation}
Recall that by Remark \ref{dotB}, 
$$ |\dot{B}_1(t)| \le C |\dot{y}(t)|. $$

Combining the above estimates and  the fact that $\ddot{I}(t) \to +\infty$, when $t \to 0$, we get
$$ |\frac{d^3 I}{d t^3}| \le C_{10} \ddot{I}^{\frac{\al+1}{\al}} \ddot{I}^{\frac{1}{2}} = C_{10} \ddot{I}^{\frac{3}{2} + \frac{1}{\al}} := e \ddot{I}^{\gm}. $$

\end{proof}

Put $y(t)$ in polar coordinates: $y(t) = \rho(t) e^{i \tht(t)}.$ We want to know how $\tht(t)$ behaves as $y(t)$ goes to a collision.

\begin{prop}
 \label{angular limit} There are finite $\tht_-, \tht_+$, such that
 $$ \lim_{t \to 0^-} \tht(t) = \tht_- , \quad \lim_{t \to 0^+} \tht(t) = \tht_+, $$
 and 
 $$ \lim_{t \to 0^{\pm}} \dot{\tht}(t) =0. $$
\end{prop}

\begin{proof}
By the definition of angular momentum
$$ \dot{J}= \frac{d}{dt}(y \times \dot{y}) = y \times \ddot{y}. $$
On the other hand for any $t \in (0, \dl),$ $y(t)$ is a solution of \eqref{ncenter}, so
$$ \ddot{y} =- \frac{m_1y}{|y|^{\al+2}} - \sum_{j =2}^{N} \frac{m_i(y-c_j)}{|y-c_j|^{\al +2}} . $$
Therefore
$$\dot{J} = y \times \ddot{y} =  \sum_{j=2}^{N} \frac{m_j y \times c_j}{|y - c_j|^{\al+2}}, $$
$$ |\dot{J}| = |y \times \ddot{y} |  \le  \sum_{j=2}^{N} \frac{m_j|c_j| \cdot |y|}{|y - c_j|^{\al+2}} \le C_1 |y| = C_1 I^{\frac{1}{2}}. $$
By the first result of Proposition \ref{asmp},
$$ |\dot{J}(t)| \le C_1 I^{\frac{1}{2}}(t) \le C_2 t^{\frac{2}{2 + \al}}. $$
Combining this with Lemma \ref{J0}, we get
$$ |J(t)| \le \int_{0}^{t} |\dot{J}(s)| \,ds \le C_2 t^{\frac{4 + \al}{2 + \al}}. $$
Hence $$ |J(t)| = \rho^2(t) |\dot{\tht}(t)| = I(t) |\dot{\tht}(t)| \le C_2 t^{\frac{4 + \al}{2 + \al}}. $$
Again by Proposition \ref{asmp}, we have $I(t) \sim C_3 t^{\frac{4}{2 + \al}},$ so
$$|\dot{\tht}(t)| \le C_4 t^{\frac{\al}{2 + \al}}. $$
Therefore
\begin{equation} \label{281}
 \lim_{ t \to 0^+} \dot{\tht}(t)= 0,
\end{equation}
by this it is easy to see there must be a finite $\tht_+$ with 
\begin{equation} \label{282}
 \lim_{t \to 0^+} \tht(t) = \tht_+,
\end{equation}
for some $\tht_+ \in [0, 2 \pi].$
The remaining results can be proven similarly.
\end{proof}

\section{The Blow-up Technique} \label{blow up}

Throughout this section we fix an arbitrary simple isolated collision solution $y(t), t \in [t_0, -\dl, t_0+\dl]$ as in Definition \ref{coll solution} and we will show that such a $y$ cannot be a minimizer of the fixed end problem under certain topological constraints.

The idea is to combining results of Theorem \ref{thm 3} and the \emph{blow-up} technique first introduced by S. Terracini in the studying of the $N$-body problem, see \cite{Ve02}, \cite{FT04}. We show that the same idea applies to the $N$-center problem as well.

Again to simplify notations, we assume $t_0=0$ and $y(0) =c_1 =0$ for the rest of the section. For any $\lmd >0$, we introduce a new $N$-center problem, which is a blow-up or rescaling of the original one with respect to $c_1$. To be precise, for each $\lmd$, we replace the positions of the $N$ centers at
$$ \cl_j = \lmd^{-\aly}( c_j-c_1)= \lmd^{-\aly}c_j, \quad \forall  j = 1 \dots, N,$$ 
here the blow-up is defined with respect to $c_1$ and $\cl_1$ is still at the origin as we assumed $c_1 =0$. A blow-up with respect to an arbitrary center $c_k$ can be defined similarly. 

Under the gravitational field of the new $N$ centers, the motion of a test particle satisfies the following equation:
\begin{equation}
 \label{rescaled} \ddot{x}(t) = \nabla \vlm(x(t)) = - \sum_{j = 1}^{N} \frac{m_j x(t)}{|x(t)- \cl_j|^{\al+2}},
\end{equation}
where $ \vlm(x(t)): = \sum_{j = 1}^N \frac{m_j}{\al |x(t) - \cl_j|^{\al}}. $

The corresponding Lagrangian and action functional are denoted by
$$ \llm (x, \dot{x}) = \ey |\dot{x}|^2 + \vlm(x),$$
$$ \alm([T_1, T_2];x) = \int_{T_1}^{T_2} \llm(x(t), \dot{x}(t)) \,dt,$$
$$ \alm_T(x) = \alm([-T, T];x). $$

\begin{dfi}
 \label{lmd rescale} Given a $\lmd>0$ and a curve $x: [a, b] \to \cc $, we define its \textbf{$\lmd$-rescale (with respect to $c_1$)} as:
 $$ \quad \xl(t) := \la [x(\lmd t)-c_1]+c_1 =\la x(\lmd t), \quad \forall t \in [\frac{a}{\lmd}, \frac{b}{\lmd}]. $$
\end{dfi}

\begin{lm}
 \label{lmd 1}
 $ A^{\lmd}([\frac{a}{\lmd}, \frac{b}{\lmd}],\xl) = \la A([a,b],x).$
\end{lm}
\begin{proof}
 By straight forward calculation.
\end{proof}

It is easy to see, $y_{\lmd}(t),$ $ t \in  [-\frac{\dl}{\lmd}, \frac{\dl}{\lmd}],$ the $\lmd$-rescale of $y$, is an isolated collision solution for the rescaled $N$-center problem.

When $\lmd$ goes to $0$, the rescaled $N$-center problem becomes more and more similar to the Kepler-type problem, as all the centers except $c_1$, were pushed further and further away and we will show that $\{\yl\}$ converges a collision-ejection solution of the Kepler-type problem.

Under the polar coordinates $y(t) = \rho(t) e^{i \tht(t)}$, recall that by Proposition \ref{angular limit}
\begin{equation}
 \label{eq ang limit} 
 \lim_{ t \to 0^{\pm}} \tht(t) = \tht_{\pm},
\end{equation}

We define the following parabolic collision-ejection solution of the Kepler-type problem.
\begin{equation} \label{ybar}
 \yb(t) =
\begin{cases}
  (\mu |t|)^{\frac{2}{\al+2}} e^{i \tht_+}  & \text{ if } t \in [0, +\infty), \\
  (\mu |t|)^{\frac{2}{\al+2}} e^{i \tht_-} & \text{ if } t \in (-\infty, 0],
\end{cases}
\end{equation}
where $\mu = (\al +2) (\frac{m_1}{2 \al})^{\frac{1}{2}}. $

\begin{prop}
 \label{blowup limit}
 For any $T> 0$, there is a sequence of positive numbers $\{ \lmd_n \searrow 0\}$, such that
 \begin{enumerate}
  \item $y_{\lmd_n}(t)$ converges uniformly to $\yb(t)$ on $[-T,T]$;
  \item $\dot{y}_{\lmd_n}(t)$ converges uniformly to $\dot{\yb}(t)$ on any compact subset of $[-T, T] \setminus \{0\}. $
 \end{enumerate}

\end{prop}

\begin{proof}
We fix an arbitrary $T>0$ for the rest of the proof. First we will show the desired convergences are pointwise.

Recall that $\yl(t) = \la y(\lmd t)$ and $y(\lmd t) = \rho(\lmd t) e^{i \tht(\lmd t)}$, then for any $t>0$
\begin{align*}
 \yl(t) &= \la |y(\lmd t)| \frac{y(\lmd t)} {|y(\lmd t)|} = \la |y(\lmd t)| e ^{ i \tht(\lmd t)} \\
        &= \la \frac{|y(\lmd t)|}{(\mu \lmd t)^{\power}}(\mu \lmd t)^{\power} e^{i \tht(\lmd t)}\\
        &= \frac{|y(\lmd t)|}{(\mu \lmd t)^{\power}} (\mu t)^{\power} e^{i \tht(\lmd t)}.
\end{align*}

Since $ | y(\lmd t)| = I^{\frac{1}{2}}(\lmd t) \sim (\mu \lmd t)^{\power},$ and $\lim_{t \to 0^+} \tht(t) = \tht_+$, therefore for any $t \in (0, T]$
$$ \lim_{\lmd \to 0} \yl(t) = (\mu t)^{\power}e^{i \tht_+} = \yb(t). $$
Similarly for any $t \in [-T, 0)$
$$ \lim_{\lmd \to 0} \yl(t) = (\mu |t|)^{\power} e^{i \tht_-} = \yb(t). $$
At the same time $\yl(0) \equiv 0$ for any $\lmd$. Hence $\yl$ converges to $\yb$ point-wisely as $\lmd$ goes to zero.

Now let's try to estimate $\dot{y}_{\lmd}(t)$. For any $t > 0$,
\begin{align*}
 \dot{y}_{\lmd}(t) & = \lmd^{\powerr}\dot{y}(\lmd t) = \lmd^{\powerr} [ \dot{\rho}(\lt)e^{i \tht(\lt)} + i \rho(\lt) \dot{\tht}(\lt) e^{i \tht(\lt)}]\\
	    & = \lmd^{\powerr}[\frac{1}{2} \dot{I}(\lt)I^{-\frac{1}{2}}(\lt) + i I^{\frac{1}{2}}(\lt) \dot{\tht}(\lmd t)]e^{i \tht(\lt)} \\
	    & = \big[ \frac{1}{2} \frac{\dot{I}(\lt) I^{-\frac{1}{2}}(\lt)}{(\mu \lt)^{-\frac{\al}{2 + \al}}}(\mu t)^{-\powerr} + i \frac{I^{\frac{1}{2}}(\lt)}{(\mu \lt) ^{\power}}\lmd(\mu t)^{\power} \dot{\tht}(\lt) \big] e^{i \tht(\lt)}.
\end{align*}
On the other hand, by Proposition \ref{asmp}
$$ \dot{I}(\lt) I^{-\frac{1}{2}}(\lt) \sim \frac{4}{2 + \al} \mu (\mu \lt)^{-\frac{\al}{2 + \al}},$$
$$ I^{\frac{1}{2}}(\lt) \sim (\mu \lt)^{\power}. $$
Therefore for $t>0,$
\begin{align*}
 \lim_{\lmd \to 0} \dot{y}_{\lmd}(t) & = \lim_{\lmd \to 0} [\frac{2}{2 + \al}\mu (\mu t)^{-\powerr}e^{\tht(\lt)} + i \lmd(\mu t) ^{\power} \dot{\tht}(\lt)e^{i \tht(\lt)}]\\
      & = \frac{2 \mu}{2 + \al}(\mu t)^{- \powerr}e^{i \tht_+} + \lim_{\lmd \to 0} i \lmd (\mu t ) ^{\power} \dot{\tht}(\lt)e^{i \tht(\lt)} \\
      & = \frac{2 \mu}{2 + \al}(\mu t)^{- \powerr}e^{i \tht_+} = \dot{\yb}(t).
\end{align*}
Notice that \eqref{281} and \eqref{282} imply
$$ \lim_{\lmd \to 0} i \lmd (\mu t ) ^{\power} \dot{\tht}(\lt)e^{i \tht(\lt)} =0.$$
Similarly if $t < 0$, we have
 $$ \lim_{\lmd \to 0} \dot{y}_{\lmd}(t) =  \frac{2 \mu}{2 + \al}(\mu |t|)^{- \power}e^{i \tht_-} = \dot{\yb}(t).$$
Therefore
$\dot{y}_{\lmd}(t)$ converges point-wisely to $\dot{\yb}(t)$ for any $t \ne 0$.

By the definition of $y_{\lmd}$ and Proposition \ref{asmp}, there is a $\lmd^* >0$, such that for any $\lmd \in (0, \lmd^*)$
$$ |y_{\lmd}(t) | \le C_1 |t|^{\frac{2}{2+\al}}, \quad \forall t \in [-T, T],$$
$$ |\dot{y}_{\lmd}(t)| \le C_2 |t|^{-\frac{\al}{2+\al}}, \quad \forall t \in [-T, T] \setminus \{0\}.$$
As a result 
$$ \int_{-T}^T |y_{\lmd}(t)|^2 \, dt \le C_3 T^{\frac{6 +\al}{2+\al}}, $$
$$ \int_{-T}^T |\dot{y}_{\lmd}(t)|^2 \, dt \le C_4 T^{\frac{2-\al}{2+\al}}, $$
which means $\{ \| y_{\lmd}\|_{H^1}:  \lmd \in (0, \lmd^*) \}$ has a finite upper bound. Hence there is a sequence $\{ \lmd_n \searrow 0 \}$ such that $y_{\lmn}$ converges uniformly to $\bar{y}$ on $[-T,T]$ as $n$ goes to infinity.

To show that there is a sequence $\{ \dot{y}_{\lmn}\}$ converges uniformly to $\dot{\bar{y}}$ on any compact subset of $[-T, T] \setminus \{0 \}$, using the above reasoning and the Cantor diagonalization argument, it is enough to show that for any $\ep>0 $ and $\lmd^*>0$ small enough, there is a positive $C(\ep)>0$ such that 
\begin{equation}
\int_{-T}^{-\ep} + \int_{\ep}^T | \ddot{y}_{\lmd}(t) |^2 \,dt \le C(\ep), \quad \forall \lmd \in (0, \lmd^*). \label{ddoty}
\end{equation}

By a straight forward computation
$$ \ddot{y}_{\lmd}(t) = \lmd^{\frac{2 +2 \al}{2 +\al}}[\ddot{\rho}(\lmd t)- \rho(\lmd t)\dot{\tht}^2 (\lmd t) + 2 i \dot{\rho}(\lmd t )\dot{\tht}(\lmd t) +i \rho(\lmd t) \ddot{\tht}(\lmd t)] e^{i \tht(\lmd t)}.$$
Obviously $\rho(t) \dot{\tht}^2(t) \to 0$ as $t \to 0$ and by the proof of Proposition \ref{angular limit}
\[ |\dot{J}(\lmd t)| = \rho(\lmd t) | 2 \dot{\rho}(\lmd t)\dot{\tht}(\lmd t)+ \rho(\lmd t) \ddot{\tht}(\lmd t)| \le C_5 \rho(\lmd t). \]
Hence 
\[ |\ddot{y}_{\lmd}(t)| \le \lmd^{\frac{2 + 2 \al}{2 +\al}} \ddot{\rho}(\lmd t) + C_6. \]
By Proposition \ref{asmp}, for $\lmd, t$ small enough
\[ \ddot{\rho}(\lmd t) \le C_7 (\lmd t)^{-\frac{2 + 2 \al}{2 + \al}}, \]
and 
\[ |\ddot{y}_{\lmd}(t)| \le C_7 t^{-\frac{2 +2 \al}{2 +\al}} + C_6.\]
This obviously implies \eqref{ddoty}, for $ep>0$ and $\lmd^*>0$ small enough.

\end{proof}

Choose a $r>0$ small enough, such that $B(0, r) \cap \mc{C} = \{0 \},$ where $B(0,r) = \{z \in \cc: |z| \le r \},$ we further assume  $y([-\dl, \dl]) \subset B(0,r).$ 

Without loss of generality, for $\tht_{\pm}$ given in \eqref{eq ang limit}, we assume 
$$ \tht_- \in [0, 2\pi) \text{ and } |\tht_+ - \tht_-| \le 2 \pi. $$
As a result, there are two different choices of $\tht_+$:
$$\tht_+ \in [\tht_-, \tht_- + 2 \pi] \text{ or } \tht_+ \in [\tht_- - 2\pi, \tht_-].$$
Correspondingly we define the following two classes of admissible curves.

\begin{dfi} \label{gm dl}
 When $\tht_+ \in [\tht_-, \tht_-+ 2\pi]$, we define
 $$ \Gm^{*,+}_{\dl}(y)= \{ x \in H^1([-\dl, \dl], B(0,r) \setminus \{0 \}): x \text{ satisfies the following conditions} \}$$
 \begin{enumerate}
  \item $x(\pm \dl) = y(\pm \dl);$
  \item $ \text{Arg}(x(\pm \dl)) = \text{Arg}(y(\pm \dl)); $
 \end{enumerate}
 The weak closure of $\Gm^{*,+}_{\dl}(y)$ in $H^1([-\dl, \dl], \cc)$ will be denoted by $\gmp$.

 When $\tht_+ \in [\tht_- - 2 \pi, \tht_-]$, we define $\Gm^{*, -}_{\dl}(y)$ and $\gmm$ similarly.
\end{dfi}

\begin{prop}
 \label{deform} If $\tht_+ \in [\tht_-, \tht_- + 2 \pi],$ there is a $\xi \in \Gm_{\dl}^{*, +}(y)$, such that
 \begin{enumerate}
  \item $A_{\dl}(\xi) < A_{\dl}(y)$, when $\al \in (1,2).$
  \item $A_{\dl}(\xi) < A_{\dl}(y)$, when $\al=1$ and $|\tht_+ - \tht_-| < 2 \pi.$
 \end{enumerate}
If $\tht_+ \in [\tht_- - 2\pi, \tht_-]$, similar results as above hold for some $\eta \in \Gm_{\dl}^{*, -}(y)$.
\end{prop}

\begin{proof}
	
	We will only give the detail for the existence of $\xi$, while the other is exactly the same. 
	
	By Lemma \ref{lmd 1}, it is enough to show that for some $\lmd>0$ small enough, there is a $\xi_{\lmd} \in H^1([-\frac{\dl}{\lmd}, \frac{\dl}{\lmd}], B(0, \lmd^{-\frac{2}{2+\al}}r) \setminus \{0 \})$ satisfies 
	$$ \xi_{\lmd}(\pm \frac{\dl}{\lmd}) = y_{\lmd}(\pm \frac{\dl}{\lmd}); \quad \text{Arg}(\xi_{\lmd}(\pm \frac{\dl}{\lmd})) = \text{Arg} (y_{\lmd}(\pm \frac{\dl}{\lmd}))$$
	and 
	$$ A^{\lmd}_{\dl / \lmd}(\xi_{\lmd}) < A^{\lmd}_{\dl / \lmd}(y_{\lmd}). $$
	As $\xi(t) = \lmd^{\frac{2}{2+\al}} \xi_{\lmd}(\frac{t}{\lmd}),$ $t \in [-\dl, \dl]$ will satisfy all the desired requirements. 
	
	This will be demonstrated in two steps first we show this can done for the Kepler-type action functional $\ab$, then for the action functional $A^{\lmd}$ as well.

 Choose a $T>0$, by Proposition \ref{blowup limit} and the lower semicontinuity of $\ab$, we can find a sequence of $\{ \lmn \} $ converging to $0$, satisfying
 \begin{equation}
  \label{deform1} \lim_{ n \to +\infty} \ab_T(\yln) \ge \ab_T(\yb).
 \end{equation}

 Again by Proposition \ref{blowup limit}, $\yln(t)$ and $\dot{y}_{\lmn}(t)$ converges uniformly to $\yb(t)$ and $\dot{\yb}(t)$ correspondingly on the compact set  $[-T, -\ey T] \cup [\ey T, T]$. Hence there is a sequence of positive numbers $\{ \dl_n \}$ converging to $0$, such that for $n$ large enough
 $$ |\yln(t) - \yb(t)| \le \dl_n \text{ and } | \dot{y}_{\lmn}(t) - \dot{\yb}(t)| \le \dl_n, \forall t \in [-T, -\ey T] \cup [\ey T, T]. $$

 Now define a new sequence of curves $\{ \zln \}$ by
 $$ \zln(t) =
 \begin{cases}
  \frac{t-T}{\dl_n} [\yb(t) -\yln(t) ] + \yln(t), & \text{ if } t \in [T- \dl_n, T] \\
  \yb(t), & \text{ if } t \in [-T + \dl_n, T- \dl_n] \\
  \frac{t + T+ \dl_n}{\dl_n} [ \yln(t)- \yb(t) ] + \yb(t), & \text{ if } t \in [-T,-T + \dl_n].
   \end{cases}
 $$
 As a result, for $n$ large enough
  $$ \zln(\pm T) = \yln(\pm T) ; \quad   \text{Arg}(\zln(\pm T)) = \text{Arg}( \yln(\pm T)) .$$
 By a simple computation
   $$ |\dot{z}_{\lmn}(t)| \le C_1, \quad \forall t \in [-T, -T+ \dl_n] \cup [T- \dl_n, T]. $$
 Therefore
 $$ \int_{T- \dl_n}^T \ey | \dot{z}_{\lmn}|^2 + \int_{-T}^{-T+ \dl_n}\ey | \dot{z}_{\lmn}|^2 \le C_2 \dl_n. $$
 On the other hand it is easy to see,
 $$ \vb(\zln(t)) \le C_3, \quad \forall t \in [-T,-T+ \dl_n] \cup [T- \dl_n, T]$$
 Hence
 $$ \int_{-T}^{-T + \dl_n}  \vb(\zln) + \int_{T- \dl_n}^T \vb(\zln) \le 2 C_3 \dl_n.$$
 By the above argument,
 \begin{equation}
  \label{deform2}  \lim_{n \to +\infty} \ab_{T}(\zln) = \ab_T(\yb).
 \end{equation}

 At the same time, by Theorem \ref{thm 3}, there is a $\gm \in H^1([-\frac{T}{2}, \frac{T}{2}], \cc \setminus \{0 \} )$ with $\gm(\pm \frac{T}{2}) = \yb(\pm \frac{T}{2})$ and $\text{Arg}(\gm(\pm \frac{T}{2})) = \text{Arg}(\yb(\pm \frac{T}{2}))$, such that 
 $$\ab_{\frac{T}{2}}(\yb)-  \ab_{\frac{T}{2}}(\gm) = 3 \ep >0,$$
 for some $\ep >0.$
 
 As a result we can define a new sequence of curves $\xln \in H^1([-T, T], \cc \setminus \{0 \})$ by 
 $$ \xln(t) =
 \begin{cases}
  \zln(t), & \text{ if } t \in [\frac{T}{2}, T] \\
  \gm(t), & \text{ if } t \in [-\frac{T}{2}, \frac{T}{2}] \\
  \zln(t), & \text{ if } t \in [-T,-\frac{T}{2}].
   \end{cases}
 $$
 Obviously 
 $$ \xi_{\lmn}(\pm T) = z_{\lmn}(\pm T) = y_{\lmn}(\pm T);$$
 $$ \text{Arg}(\xi_{\lmn}(\pm T)) = \text{Arg}(z_{\lmn}(\pm T)) = \text{Arg}(y_{\lmn}(\pm T)),$$
and 
 $$ \ab_T(\zln) - \ab_T(\xln) = 3 \ep.$$

 Combining this with \eqref{deform1} and \eqref{deform2}, for $n$ large enough, we have
 \begin{equation}
  \label{deform3} \ab_T(\yln) - \ab_T(\xln) \ge 2 \ep.
 \end{equation}
 This finishes the first step. 
 
 Now we will compare the values of the action functionals $\ab$ and $\alm$ on $\yln$.
 \begin{align}
 \label{deform4} \begin{split}
A^{\lmn}_{T}(\yln) & = \int_{-T}^{T} [\ey |\dot{y}_{\lmn}|^2 + V_{\lmn}(\yln) ] \\
& = \int_{-T}^{T} [\ey |\dot{y}_{\lmn}|^2 + \frac{m_1}{\al|\yln|^{\al}} ] + \int_{-(T+\dn)}^{T+\dn} \sum_{j =2}^{N} \frac{m_j}{\al |\yln - c^{\lmn}_j|^{\al}}. \\
& \ge \ab_{T}(\yln).
 \end{split}
 \end{align}
 At the same time
 \begin{align*}
 \aln_{T}(\xln) & = \int_{-T}^{T} [ \ey |\dot{z}_{\lmn}|^2 + V_{\lmn}(\xln)] \\
		    & = \int_{-T}^{T} [ \ey |\dot{z}_{\lmn}|^2 + \frac{m_1}{\al |\xln|^{\al}}] + \int_{-T}^{T} \sum_{j =2}^{N} \frac{m_j}{\al |\xln - c^{\lmn}_j|^{\al}} \\
		    & := \ab_{T}(\xln) + F_n.
 \end{align*}

 By the definition of $\xln$, it is not hard to see, $|\xln(t)|, t \in [-T,T]$ have a uniform upper bound for $n$ large enough, and $|c^{\lmd_n}_j|, j>1,$ goes to infinite as $n$ goes to infinity. This means $ \lim_{ n \to + \infty} F_n = 0.$

 Therefore when $n$ is large enough
 \begin{equation}
  \label{deform5} A^{\lmn}_{T}(\xln) \le \ab_{T}(\xln) + \ep.
 \end{equation}

 By \eqref{deform3}, \eqref{deform4} and \eqref{deform5}, we have proved 
 $$ \aln_T(\yln) - \aln_T(\xln) \ge \ep, $$
 for $n$ large enough. 
 
 It is not hard to see for $n$ large enough, we have  $T < \frac{\dl}{\lmd_n}$ and $|\xi_{\lmn}(t)| \le \lmn^{-\frac{2}{2+\al}} r$, for any $t \in [-T,T]$. Since $\xln(\pm T) = \yln(\pm T),$ we can extend the domain of $\xln$ to $[-\frac{\dl}{\lmd_n}, \frac{\dl}{\lmd_n}]$ by simply attach $\yln|_{[-\frac{T}{\lmd_n}, -T]}$ and $\yln|_{[T, \frac{T}{\lmd_n}]}$ to the corresponding ends of $\xln|_{[-T,T]}$. As a result, for $n$ large enough
 
 $$ \aln_{-\dl/\lmd_n}(\yln) - \aln_{\dl / \lmd_n}(\xln) \ge  \aln_T(\yln) - \aln_T(\xln) \ge \ep >0,$$
 and $\xi_{\lmn}$ obviously satisfies all the requirements posted at the begin of the proof and we are done.

\end{proof}

\section{Proof of Theorem \ref{thm 1}} \label{sec thm 1}

This section will be devoted to the proof of Theorem \ref{thm 1}. However except the last proof, all the results in this section hold when $\al =1$ as well. 

First we will briefly recall some results from \cite{HS85}. 

\begin{dfi}
Given a generic loop $x \in \gts$, we say it has a
\begin{enumerate}
\item \textbf{singular $1$-gon }, if there is a subinterval $I \subsetneq S_T$ such that $x|_{I}$ is a sub-loop contained in the trivial free homotopy class;
\item \textbf{singular $2$-gon }, if there are two disjoint subintervals $I_1, I_2 \subsetneq S_T$, such that $x$ identifies the end points of $I_1$ with $I_2$ and the loop obtained by the identification is contained in the trivial free homotopy class.

\end{enumerate}
\end{dfi}

Then the following result was proved in Theorem $4.2$, \cite{HS85}.
\begin{thm} \label{Hass}
If a generic loop $x$ has excess self-intersection, then $x$ has a singular $1$-gon or singular $2$-gon.
\end{thm}

Several Lemmas will be established before we can prove Theorem \ref{thm 1}. 

\begin{lm}
 \label{coercive} Given an arbitrary non-trivial free homotopy class $\tau$, for any $T>0$, there is a $q \in \gtt$, such that $A_{T/2}(q) =c_T(\tau).$
\end{lm}

\begin{proof}
	For any $x \in \gts$, let $\| x \|_{\infty}$ be the supremum norm and $\| x \|_2$ the $L^2$ norm. Since $\tau$ is non-trivial, there is at least one $c_j \in \mc{C}$ such that $\text{ind}(x, c_j) \ne 0.$ It is easy to see 
 	$$ \| x \|_{\infty} \le \| x -c_j \|_{\infty} + | c_j|. $$
 	Notice that  for any $s \in S_T$, there must be a $s^* \in S_T$, such that 
 	$$ |x(s) - c_j | \le |x(s) - x(s^*) |. $$
 	At the same time 
 	$$ |x(s) - x(s^*)| \le \int_{S_T} |\dot{x}(t)| \,dt \le T^{\ey} \| \dot{x} \|_2. $$
 	The last inequality follows from the Cauchy-Schwarz inequality. 
 	By the above inequalities we get
 	$$ \|x\|_{\infty} \le \| x- c_j\|_{\infty} + |c_j| \le T^{\ey} \|\dot{x}\|_2 + |c_j|.$$
 	Hence 
 	$$ \|x\|_2 \le T^{\ey} \| x\|_{\infty} \le T\| \dot{x} \|_2 + T^{\ey}|c_j|. $$
 	As a result 
 	$$ \|x\|_{H^1} \le  C_1 \|\dot{x}\|_2 + C_2. $$

 	The last inequality obviously implies the action functional $A_{T/2}$ is coercive on $\gts$. As $A_{T/2}$ is weakly lower semi-continuous, a standard argument in calculus of variation shows that there is a minimizing sequence of loops $\{ q_n \in \gts \}$, which converges weakly to a $q \in \gtt$ with respect to the Sobolev norm (this is also implies $q_n$ converges to $q$ with respect to the supremum norm). Therefore

     $$ \aet(q) = \liminf_{n \to \infty} \aet(q_n) = c_T(\tau). $$

\end{proof}

Given an $x \in \gtt$, we say $\Delta(x): = \{ t \in S_T: x(t) \in \mc{C} \}$ is the set of collision moments. 

\begin{lm}
	\label{isolated} If $y \in \gtt$ satisfies $\aet(y) =c_T(\tau)$, then $\Delta(y)$ is an isolated set and for any $t \in S_T \setminus \Delta(y)$, $y(t)$ is a solution of equation \eqref{ncenter} with constant energy 
	$$ \ey|\dot{y}(t)|^2 -V(y(t)) = h,$$
	for some $h$ independent of $t$. 
\end{lm}
We omit the proof here as it is a simpler version of a similar result for the N-body problem as long as we have the Lagrange-Jacobi identity, Lemma \ref{LaJa}, where a detailed proof can be find in \cite{C02}, \cite{Ve02} and \cite{FT04}. The above lemma tells us, for a minimizer, its set of collision moments must be finite. The difficulty is to show that it is actually empty. 

For that we need to show it is possible to choose generic loops with no excess self-intersection as our minimizing sequence of loops and the following lemma will the key to that.

\begin{lm}
\label{taut} Given an $x \in \gts$ with a finite $A_{T/2}(x)$, for any $\ep^*>0$, there is a generic loop $\xt \in \gts$ with no excess self-intersection satisfying 
$$ A_{T/2}(\xt) \le A_{T/2}(x) + \ep^*. $$
\end{lm}

\begin{proof}
Recall that for the Lagrange assocated with our problem, it is well known that  $A_{T/2} \in C^1(H^1(S_T, \cx), \rr)$, see \cite{AC93} page $20$.

Because $C^{\infty}(S_T, \cx)$ is a dense subset of $H^1(S_T, \cx)$ (the smooth approximation theorem, \cite{E98} page 252), the space of all smooth immersion $T$-periodic loops is dense in $C^{\infty}(S_T, \cx)$ (Theorem $2.2.12$, \cite{Hi76} page 53), and the space of all smooth immersion $T$-periodic loops in \emph{general position} is dense in the space of all smooth immersion $T$-periodic loops (Exercise $2$, \cite{Hi76} page 82), for any $\ep>0$, we can find a generic loop $\eta \in \gts$ with $A_{T/2}(\eta) \le A_{T/2}(x) + \ep. $

If $\eta$ does not have any excess self-intersection, let $\tilde{x}=\eta$. If it does, by Theorem \ref{Hass}, there must be a singular $1$-gon or singular $2$-gon. 

If there is a singular $1$-gon, say $[t_1, t_2] \subset S_T$ with $\eta(t_1) = \eta(t_2)$, we can just reverse the time on $\eta|_{[t_1, t_2]}$. Obviously the new loop is still in the same free homotopy class. Afterwards $\eta(t_1)=\eta(t_2)$ becomes a non-transverse intersection. Using the approximation results given before we can find a generic loop from $\gts$ in a small enough neighborhood of $\eta$, which eliminates the non-transverse self-intersection without increasing the number of transverse self-intersection and at the same time the action functional is increased by no more than $\ep$.  

If there is a singular $2$-gon, then there are two disjoint subintervals  $[s_1, s_2]$, $[t_1, t_2]$ of $S_T$, with one of the following possibilities
$$\eta(s_1) = \eta(t_1), \eta(s_2) = \eta(t_2) \text{ or } \eta(s_1) = \eta(t_2), \eta(s_2) = \eta(t_1).$$
If it is the first case we make a change as indicated in Figure \ref{N4}. If it is the second case we make a change as indicated in Figure \ref{N5}. In either case we get a loop still contained in the original free homotopy class with two non-transverse self-intersections. Then just repeat the same argument used in the previous case. 

\begin{figure}
	\centering
	\includegraphics[scale=0.70]{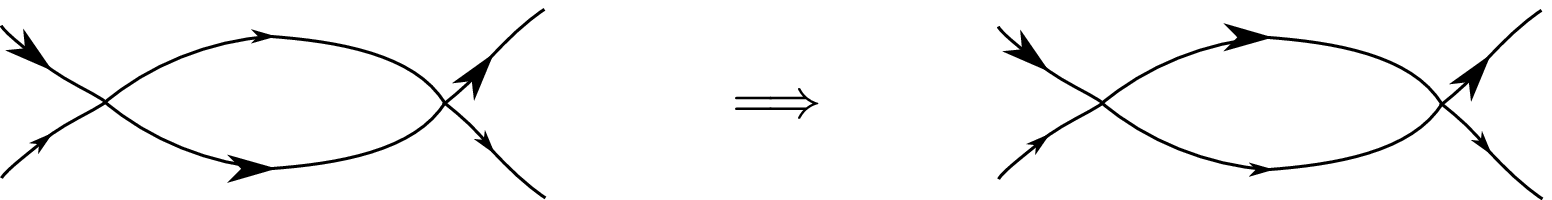}
	\caption{}
	\label{N4}
\end{figure}

\begin{figure}
	\centering
	\includegraphics[scale=0.70]{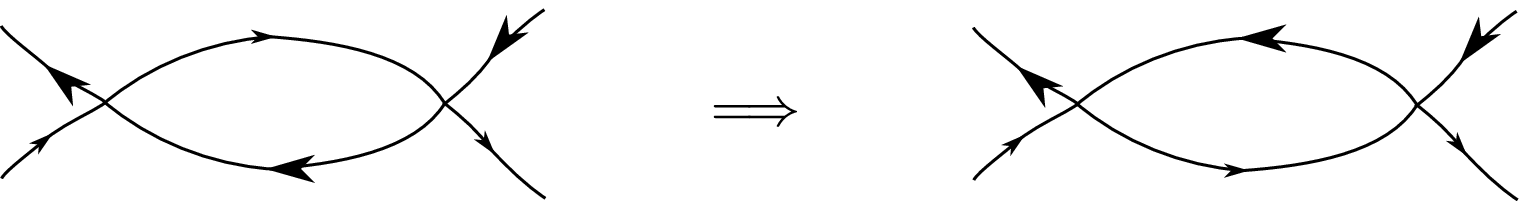}
	\caption{}
	\label{N5}
\end{figure}

Obviously after a finite steps as above we can eliminate all the singular $1$-gons and singular $2$-gons and get a generic loop $\tilde{x} \in \gts$ with no excess self-intersection satisfying 
$$ A_{T/2}(\tilde{x}) \le A_{T/2}(x) + k \ep,$$
for some finite $k \in \zz^+$. Let $\ep = \ep^*/k$ and we are done. 
\end{proof}

Now we are ready to prove Theorem \ref{thm 1}.
\begin{proof}

 [Proof of Theorem \ref{thm 1}] 
 Let $\{q^*_n \in \gts \}$ be a minimizing sequence of loops, i.e., $c_T(\tau) = \lim_{n \to \infty} A_{T/2}(q^*_n)$, and $\{\ep_n\}$ be a sequence of positive constants which converges to zero, then by Lemma \ref{taut} there is a sequence of generic loops $\{q_n \in \gts \}$ with no excess self-intersections satisfying 
 $$ \aet(q_n) \le \aet(q^*_n) + \ep_n. $$
 As a result 
 $$\lim_{n \to \infty} A_{\frac{T}{2}}(q_n) =c_T(\tau).$$ 
 From the proof of Lemma \ref{coercive}, we know after passing $\{q_n\}$ to a subsequence, it converges weakly (Sobolev norm) and uniformly (supremum norm) to a $q \in \gtt$ with 
 $$ \aet(q) = \lim_{n\to \infty} \aet(q_n) = c_T(\tau). $$
 
 Let $\Delta(q)$ be the set of collision moments, we only need to show this set is empty. 
 
 Assume $\Delta(q) \ne \emptyset,$ Lemma \ref{isolated} tells us it must be isolated. Without loss of generality, assume $q(0) = c_1= 0$. Choose an $r>0$ small enough such that $c_1=0$ is the only center contained in the closed disk $B(0, 2r).$ By Proposition \ref{asmp}, we can find a $\dl>0$ small enough such that $q|_{[-2\dl, 2\dl]}$ is an isolated collision solution with $ q([-2\dl, 2\dl]) \subset B(0, r)$. As $q_n$ converges uniformly to $q$, for $n$ large enough we have $q_n([-2\dl, 2\dl]) \subset B(0, 2r)$. 

 We claim for $n$ large enough $q_n|_{[-2\dl, 2\dl]}$ does not have any self-intersection. If not, there must be a subinterval $I \subset [-2\dl, 2\dl]$, such that $q_n|_{I}$ is an innermost sub-loop. As $ q_n(I) \subset B(0, 2r)$, at most one center, $c_1$, can be enclosed by this innermost sub-loop, which is a contradiction to the fact that $\tau$ is an admissible class. 

 Since $q_n$ converges uniformly to $q$, $q|_{[-2\dl, 2\dl]}$ can not have any transverse self-intersection as well. By Defintion \ref{coll solution}, $q|_{[-2\dl, 2\dl]}$ is a simple isolated collision solution. Put $q|_{[-2\dl, 2\dl]}$ in polar coordinates $q(t) = \rho(t)e^{i \tht(t)},$  by Proposition \ref{angular limit}, there are two finite angles $\tht_-, \tht_+$, such that, $ \lim_{t \to 0^{\pm}} \tht(t) = \tht_{\pm}.$ Furthermore we may assume $|\tht_+ -\tht_-| \le 2 \pi$, as we only define the polar coordinates on $q|_{[-2\dl, 2\dl]}$ not the entire loop $q|_{S_T}.$

 If $\tht_+ \in [\tht_-, \tht_- + 2\pi]$ (resp. $\tht_+ \in [\tht_- -2\pi, \tht_-]$), by Proposition \ref{deform}, there is a $\xi^+ \in \Gamma^{*, +}_{\dl}(q)$ (resp. $\xi^- \in \Gm^{*, -}_{\dl}(q)$), where $\Gamma^{*, +}_{\dl}(q)$ (resp.  $\Gm^{*, -}_{\dl}(q)$ ) is defined as in Definition \ref{gm dl} (with $y$ replace by $q$) satisfying 
 $$ A_{\dl}(\xi^+) < A_{\dl}(q) (\text{resp. } A_{\dl}(\xi^-) < A_{\dl}(q)). $$

 According to Definition \ref{gm dl}, $\xi^+$ and $\xi^-$ have the same initial and end points as $q|_{[-\dl, \dl]}$, so we can define two new loops as 
 $$ q^+(t) =
 \begin{cases}
 \xi^+(t), & \text{ if } t \in [-\dl, \dl] \\
  q(t), & \text{ if } t \in S_T \setminus [-\dl, \dl],
 \end{cases}
 $$
 $$ q^-(t) =
 \begin{cases}
 \xi^-(t), & \text{ if } t \in [-\dl, \dl] \\
  q(t), & \text{ if } t \in S_T \setminus [-\dl, \dl].
 \end{cases}
 $$
 Obviously the actional functional of both loops, $q^+, q^-$, are strictly less than $q$. Therefore if we can show that one of them is contained in $\gtt$, it will give us a contradiction and we are done. 

 Choose a $0 < \dl^* < \dl$, as $q|_{[-2\dl, -2\dl+ \dl^*]}$ (resp. $q|_{[2\dl-\dl^*, 2\dl]}$) is a solution of equation \eqref{ncenter} with a positive distance from any centers and $q_n$ converges uniformly to $q$ as $n$ goes to infinity, by the standard result of calculus of variation, when $\dl^*$ is small enough, for each $n$ large enough, the action functional has a unique minimizer $\eta^-_n$ (resp. $\eta^+_n$) among all Sobolev curves defined on $[-2\dl, -2\dl+\dl^*]$ (resp. $[2\dl-\dl^*, 2\dl]$) with the same initial point $q_n(-2\dl)$ (resp. $q_n(2\dl-\dl^*)$) and the same end point $q(-2\dl+ \dl^*)$ (resp. $q_n(2\dl)$). Obviously each $\eta^-_n$ (resp. $\eta^+_n$) is a solution of equation \eqref{ncenter} and it converges to $q|_{[-2\dl, -2\dl +\dl^*]}$ (resp. $q|_{[2\dl-\dl^*, 2\dl]}$) with respect to the $C^1$ norm as $n$ goes to infinity. 

 Now for each $n$ (large enough), we can define a curve $\gm^+_n$ by
  $$ \gm^+_n(t) =
 \begin{cases}
  \eta^-_n(t), & \text{ if } t \in [-2\dl, -2\dl+\dl^*] \\
  q(t), & \text{ if } t \in [-2\dl+\dl^*, -\dl] \\
  \xi^+(t), & \text{ if } t \in [-\dl, \dl] \\
  q(t), & \text{ if } t \in [\dl, 2\dl-\dl^*]\\
  \eta^+_n(t), & \text{ if } t \in [2\dl-\dl^*, 2\dl],
   \end{cases}
 $$
 and a curve $\gm^-_n$ just as above only with $\xi^+$ replaced by $\xi^-$. Using these curves we further define two loops $\qt^+_n, \qt^-_n \in H^1(S_T, \cx)$ by
 $$ \qt^+_n(t) =
 \begin{cases}
 \gm^+_n(t), & \text{ if } t \in [-2\dl, 2\dl] \\
  q_n(t), & \text{ if } t \in S_T \setminus [-2\dl, 2\dl],
 \end{cases}
 $$
 $$ \qt^-_n(t) =
 \begin{cases}
 \gm^-_n(t), & \text{ if } t \in [-2\dl, 2\dl] \\
  q_n(t), & \text{ if } t \in S_T \setminus [-2\dl, 2\dl].
 \end{cases}
 $$

 We can do this beacuase $\gm^+_n$ and $\gm^-_n$ both have the same initial and end points as $q_n|_{[-2\dl, 2\dl]}$. Since $q|_{[-2\dl,2\dl]}$ does not have any transverse self-intersection and so is each $q_n|_{[-2\dl, 2\dl]}$, when $n$ is large enough, either $\gm^+_n$ or $\gm^-_n$ must be homotopic to $q_n|_{[-2\dl, 2\dl]}$ inside $B(0, 2r) \setminus \{0\}$ with both end points fixed. This means either $\qt^+_n$ or $\qt^-_n$ must be contained in $\gtt$. Without loss of generality we may assume this is true for a subsequence of $\qt^+_n$. By the way we define each $\qt^+_n$, it is not hard to see it converges weakly to $q^+$, which means $q^+ \in \gtt$ and this finishes our proof.

\end{proof}

\begin{rem} \label{rm thm 1}
In the Newtonian case, $\al=1$, it is not hard to see the above contradictory argument fails only when we have
\[ |\tht_+ - \tht_-| = 0 (\text{mod } 2 \pi). \]
This is due to the extra condition required in Proposition \ref{deform} (essentially in Theorem \ref{thm 3}), when $\al =1$. 
\end{rem}

\section{ Proof of Theorem \ref{thm 2}}

In this section, we always assume the potential is Newtonian, i.e., $\al =1.$ Let $y(t), t \in [t_0 -\dl, t_0 + \dl]$ be an isolated collision solution with $y(t_0)= c_k \in \mc{C}$. Set $y(t) - c_k = \rho(t) e^{i \tht(t)}$, by Proposition \ref{angular limit}, the following limits exist  $  \lim_{t \to t_0^{\pm}} \tht(t) = \tht_{\pm}(t_0).$ As explained in Remark \ref{rm thm 1}, the key point is to understand what happens when
$$ |\tht_+(t_0) - \tht_-(t_0)| = 0 (\text{mod } 2 \pi). $$

A similar situation was discussed in \cite{ST12}. Following their idea we will perform a local Levi-Civita regularization near the isolated collision and show that under the above condition $y$ must be a \emph{collision-reflection solution}. 
\begin{dfi}
 \label{coll ref} We say $y$ is a \textbf{collision-reflection solution}, if the following is true
 $$ y(t_0 +t) = y(t_0 -t), \quad \forall t \in [0,\dl].$$
\end{dfi}

Let's recall the basics of a Levi-Civita transformation. We introduce a new coordinates $z$ with  
$$ z^2(t) = y(t) - c_k;$$
and a new time parameter $s = s(t), t \in [t_0-\dl, t_0 + \dl]$ with 
$$ ds = |y(t)-c_k|^{-1} dt. $$
To distinguish between the two time parameters, we set $z'(s) =\frac{d}{ds} z(s). $ By Proposition \ref{asmp}, a simple estimate shows that $$\int_{t_0-\dl}^{t_0 + \dl} |y(t)-c_k|^{-1} dt < \infty .$$ 
As a result, we may assume $s(t_0)= 0$ and there are two finite numbers $S^{\pm}>0$, such that 
$$ S^+ = s(t_0 + \dl); \quad S^- = -s(t_0 -\dl).$$
Using the factor that the energy of $y$ is a constant, i.e., there is a constant $h$, such that  
 \begin{equation} \label{energy}
 \ey |\dot{y}(t)|^2 - V(y(t)) = h, \quad \forall t \in [t_0-\dl, t_0 + \dl] \setminus \{t_0\}. \end{equation} 
 A direct computation shows that $z(s), s \in [-S^-, S^+] \setminus \{0 \}$ is a solution of the following equation
\begin{equation}
 \label{reg 1}
 2 z'' = hz + z V_k(z^2 +c_k) + |z|^2 \bar{z} \nabla_y V_k(z^2 + c_k)
\end{equation}
where $ V_k(z^2 +c_k) =V_k(y) = \sum_{j \ne k} \frac{m_j}{|y- c_j|}.$ Obviously the singularity caused by the collision with the center $c_j$ does not exist anymore in the above equation.

\begin{lm} \label{reflect 1}
Let $y$ and $z$ be defined as above, if $ |\tht_+(t_0) - \tht_-(t_0)| = 2 \pi $ then the following one-sided limits are well-defined and satisfying
$$ \lim_{ s \to 0^-} z'(s) = - \lim_{s \to 0^+} z'(s).$$
\end{lm}

\begin{proof}
 To simplify notation, we assume $t_0=0$, $c_k=0$ and $\tht_{\pm}(t_0) = \tht_{\pm}$. Put $z(s)$ in polar coordinates: $z(s) = r(s) e^{i \phi(s)}$, then 
 \begin{equation} \label{zprime}
 z'(s) = [r'(s)+ i r(s) \phi'(s)]e^{i \phi(s)}.
 \end{equation}
Since $z^2 = y = \rho e^{i \tht}$, we have 
$$ r(s) = \rho^{\ey}(t(s)), \quad \phi(s) = \ey \tht(t(s)). $$
Therefore by a direct computation 
$$ r'(s) = \ey \rho^{\ey}(t(s)) \dot{\rho}(t(s)), $$
$$ \phi'(s) = \ey \rho(t(s))\dot{\tht}(t(s)). $$
By Proposition \ref{asmp} and \ref{angular limit}, we can get 
$$ \rho(t) \sim (\mu t)^{\frac{2}{3}}, \quad \dot{\rho}(t) \sim \frac{2}{3} \mu (\mu t)^{-\frac{1}{3}},$$
and 
$$ \lim_{t \to 0^{\pm}} \dot{\tht}(t) = 0. $$
Recall the definition of the time parameter $s$, in particular $s(0)=0$, we have 
$$ \lim_{s \to 0^+} r'(s) = \lim_{t \to 0^+} \ey \rho^{\ey}(t) \dot{\rho}(t) = \frac{1}{3} \mu, $$
$$ \lim_{s \to 0^+} \phi'(s) = \lim_{t \to 0^+} \ey \rho(t) \dot{\tht}(t) = 0. $$
At the same time 
$$ \lim_{s \to 0^+} \phi(s) = \lim_{t \to 0^+} \ey \tht(t) = \tht_+. $$
Plugging the above limits into \eqref{zprime}
$$ \lim_{s \to 0^+} z'(s) = \frac{1}{3} \mu e^{i \frac{\tht_+}{2}}. $$
A similar calculation will show 
$$ \lim_{s \to 0^-} z'(s) = \frac{1}{3} \mu e^{i \frac{\tht_-}{2}}. $$
Now the desired result immediately follows from $|\tht_+ - \tht_-| = 2 \pi.$

\end{proof}

\begin{lm}
 \label{reflect 2} Under the same conditions of Lemma \ref{reflect 1}, $y$ is a collision-reflection solution.
\end{lm}

\begin{proof}
  Again to simplify the notation, let's assume $t_0=0$ and $y(0)=c_1=0$. Set $z^*(s) = z(-s)$ for $s >0$.  Since \eqref{reg 1} is independent of the time parameter $s$ and the first derivative, $z^*(s)$ also satisfies equation \eqref{reg 1}. 

  By Lemma \ref{reflect 1},
  $$ \lim_{s \to 0^+} (z^*)'(s) = \lim_{s \to 0^+} -z'(-s) = \lim_{s \to 0^-} - z'(s) = \lim_{s \to 0^+} z'(s).$$

  Since $z^*(0) = z(0) = 0$, the existence and uniqueness theorem of ordinary differential equation tells us
  $$z(s) = z^*(s) = z(-s), \text{ for } s >0. $$
  As a result $S^- = S^+$ and 
  $$ y(t) = y(-t), \quad \forall t \in [0, \dl].$$

 \end{proof}

 Now we are ready to prove Theorem \ref{thm 2}.

 \begin{proof}

  [Theorem \ref{thm 2}]
  Following the notations introduced in Section \ref{sec thm 1}, as we mentioned most of the results from that section still hold for the Newtonian potential. Let $ q_n $ (generic loop with no excess self-intersection) and $q$ be the same loops obtained at the beginning of the proof of Theorem \ref{thm 1} with 
  $$ A_{\frac{T}{2}}(q) = \lim_{n \to \infty} A_{\frac{T}{2}}(q_n) = c_{T}(\tau). $$
  
  If $\Delta(q) =\emptyset$, then $q$ is collision free solution of \eqref{ncenter} and we are in the first case. 

  If $\Delta(q) \ne \emptyset$, for any $t_0 \in \Delta(q)$ with $q(t_0) = c_{k_1}$, we can define a local polar coordinates near $q(t_0)= c_{k_1}$
  $$q(t) - c_{k_1} = \rho(t) e^{i \tht(t)}. $$
  By Proposition \ref{angular limit}, the following limits exist $ \tht_{\pm}(t_0) = \lim_{t \to t_0^{\pm}} \tht(t).$  As explained in Remark \ref{rm thm 1}, such a collision can exist only when 
  \begin{equation}
   \label{21} |\tht_+(t_0) - \tht_-(t_0)| =0 (\text{mod } 2 \pi).   \end{equation}
   
  Because the polar coordinates are only defined locally, we may assume \[\tht_+(t_0) - \tht_-(t_0)=2 \pi.\]  This is the condition required in Lemma \ref{reflect 1} and \ref{reflect 2}. Now by Lemma \ref{reflect 2}, there is a $\dl>0$ such that
  $$ q(t_0 +t) = q(t_0 -t) \quad \forall t \in [0, \dl].$$  
  In fact the existence and uniqueness theorem of ordinary differential equation guarantees  above identity will hold until a moment $t_1$ (set $\tb = t_1 -t_0>0$) when one of the following two situations first appears
  $$\dot{q}(t_1) = 0 \quad \text{ or } \quad q(t_1) =c_{k_2} \in \mc{C}. $$

  If $\dot{q}(t_1) = 0$, then $q(t_1)$ is a break point (with zero velocity), so under the gravitational force the test particle has to come back to $q(t_0)$ following the same path, i.e.
  $$ q(t_1 +t) = q(t_1 -t), \quad \forall t \in [0, \tb].$$
  As a result $t_1 + \tb \in \Delta(q)$ with $q(t_1 + \tb) = c_{k_1}$ and the same condition as \eqref{21} must hold at this moment as well. Therefore we can apply Lemma \ref{reflect 2} again and get
  $$ q(t_1 + \tb +t) = q(t_1 + \tb -t), \quad \forall t \in [0, \tb]. $$
  By repeating the above arguments in both directions of the time, we find that $q$ must be a degenerating loop of period $2 \tb$, which goes back and forth between $q(t_0) =c_{k_1}$ and $q(t_1)$ along the curve $q_|{[t_0, t_1]}$ without containing any other centers. Namely it is a collision-reflection solution between $c_{k_1}$ and $q(t_1)$. Which is absurd, as explained in the following.
  
  Because $\tau$ is admissible, each innermost loop of $q_n$ encloses at least two different centers. If $q_n$ converges to a curve as above then it must contain at least two different centers. 
  
  If $q(t_1)  =c_{k_2} \in \mc{C}$ (it is possible that $c_{k_2} =c_{k_1}$), then $t_1 \in \Delta(q)$ with the same condition as \eqref{21} holds again. By the same argument as above we once again can show that $q$ must be a collision-reflection solution between $c_{k_1}$ and $c_{k_2}$ with period $2\tb$. 
  
  As a result, for each $c_j \in \mc{C} \setminus \{c_{k_1}, c_{k_2}\}$, the index of $q$ with respect to $c_j$ is well defined and vanishes: $\text{Ind}(q, c_j)=0.$ 
  
  Obviously the same is also true for any $q_n$ with $n$ large enough, which means for the given free homotopy class $\tau$, we must have 
  $$ \mf{h}(\tau)_j = 0, \quad \forall j \in \{1, \dots, N \} \setminus \{k_1, k_2 \}.$$
  This finishes the entire proof.




 \end{proof}

 \section {The Kepler-type Problem} \label{sec thm 3}
This section will be devoted to the proof of Theorem \ref{thm 3}. As we mentioned the idea is to consider an obstacle minimizing problem that have been studied in \cite{TV07}, \cite{ST12} and \cite{BTV13}. In particular we will first following the approach given in \cite{ST12} to get a result of the fixed energy problem, namely Theorem \ref{collisionless}. The main difference is afterwards we will use this theorem to get a related result of the fixed time problem.

In order to relate the fixed energy problem and fixed time problem, we need to study curves defined on different time interval. Let $\xb$ be a parabolic collision-ejection solution of the Kepler-type problem as defined in Definition \ref{coej}. For any $S, T >0$, we set
$$ \Gm^*_S(\xb; T):= \{ x \in \h([-S, S], \cc \setminus \{ 0 \}): x \text{ satisfies the following conditions } \}. $$
\begin{enumerate}
 \item $x(\pm S) = \xb( \pm T);$
 \item $ \text{Arg}(x(-S)) = \phi_-$ and $\text{Arg}(x(S)) = \phi_+. $
\end{enumerate}

Let $\Gm_S(\xb; T)$ be the weak closure of $\Gm^*_S(x; T)$ in $H^1([-S, S]; \cc).$ We set
$$ \cb(\xb;T) = \inf \{ \ab_S(x): x \in \Gm_S(\xb; T) \text{ for any } S >0 \}. $$

\begin{rem} The readers should notice that 
	\begin{enumerate}
		\item  $\Gm^*_T(\xb; T)= \Gm^*_T(\xb)$ and $\Gm_T(\xb; T)= \Gm_T(\xb)$, if and only if $S=T$,
		\item $\cb(\xb; T)$ is different from $\cb_T(\xb)$ defined in the first section. In fact $\cb(\xb; T) \le \cb_T(\xb)$ and for most choices of $T$, the inequality is strict.
	
	\end{enumerate} 
\end{rem}

Instead of Theorem \ref{thm 3}, we will first prove a slightly weaker result stated as following.

\begin{thm}
 \label{thm 4}
 For any $T > 0$ and $0 < |\phi_+ - \phi_-| \le 2\pi$, there is a $S > 0$ and  $\eta \in \Gm_{S}(\xb; T)$, such that $\ab_{S}(\eta) = \cb(\xb; T). $ 
 \begin{enumerate}
  \item If $\al \in (1,2),$ then $\eta \in \Gm^*_S(\xb; T)$ is a classical solution of the Kepler-type problem with zero energy and  $ \ab_S(\eta) < \ab_T(\xb);$
  \item If $\al =1$ and $| \phi_+ - \phi_-| < 2 \pi,$ then $\eta \in \Gm^*_S(\xb; T)$ is a classical solution of the Kepler-type problem with zero energy and  $ \ab_S(\eta) < \ab_T(\xb);$
 \end{enumerate}

\end{thm}

\begin{rem}
    Theorem \ref{thm 3} shows that $\xb|_{[-T,T]}$ is not a minimizer of the fixed end problem with the same time interval and Theorem \ref{thm 4} shows that $\xb|_{[-T,T]}$ is not a minimizer of the fixed end problem with an arbitrary time interval. Therefore the result obtained in Theorem \ref{thm 4} is weaker than in Theorem \ref{thm 3}. However we will show that Theorem \ref{thm 4} actually implies Theorem \ref{thm 3}. 
\end{rem}

In Theorem \ref{thm 4}, we need to show $\xb$ is not a free-time minimizer in certain admissible class of curves. As the time interval is not fixed, it is more natural to study the \emph{Maupertuis' functional} rather than the action functional.

For a given energy $h$, the associated Maupertuis' functional of the Kepler-type problem is defined as
$$ \Mb_h([-T, T];x) := \int_{-T}^{T} \frac{1}{2} |\dot{x}|^2 \,dt \int_{-T}^T \vb(x(t)) + h \,dt, \quad \forall x \in \h([-T,T], \cc). $$

\begin{thm} \label{thm mf}
For any $a<b$, If $y \in \h([a,b], \cc \setminus \{0 \})$ is a critical point of $\Mb_h$ with $\Mb_h(y) >0$, then $y(t)$ is a classical solution of
 \begin{equation}
 \begin{cases}
  \om^2(y) \ddot{y}(t) = \nabla \vb (y(t)), & t \in [a, b], \\
  \frac{1}{2} |\dot{y}(t)|^2 - \frac{V(y(t))}{\om^2(y)} = \frac{h}{\om^2(y)}, & t \in [a, b], \\
  y(a) = q_1, y(b) = q_2,
 \end{cases}
 \end{equation}
where $\om(y)$ is a positive number defined by
\begin{equation} \label{omega}
 \om^2(y) := \frac{\int_a^b V(y(t)) + h \,dt}{ \int_a^b \frac{1}{2} | \dot{y}(t)|^2 \,dt}.
\end{equation}
Furthermore $x(t): = y(\om(y) t), t \in [\frac{a}{\om(y)}, \frac{b}{\om(y)}]$  is a classical solution of the Kepler-type problem with energy $h$, i.e., for any $t \in [a,b]$
$$ \begin{cases}
 \ddot{x}(t) = \nabla \vb (x(t)), \\
 \frac{1}{2}|\dot{x}(t)|^2 - V(x(t)) = h.
\end{cases}
$$
\end{thm}

This is a standard result, whose proof can be found in \cite{AC93}.

As a parabolic solution, $\xb(t)$'s energy is zero, we only need to study $\Mb_0$. To simplify notation, set $\Mb:= \Mb_0.$

Let
$$ \mbb_T = \inf \{ \Mb([-T,T]; x): x \in \Gm_T(\xb) \}, $$
Where $\Gm_T(\xb)$ is defined as in Section \ref{sec intro}. We will prove the following result regarding the Maupertuis' functional and obtain Theorem \ref{thm 4} as its corollary.

\begin{thm} \label{collisionless}
 For any $T>0$ and $ 0 < |\phi_+ - \phi_-| \le 2 \pi$, there is a $y \in \Gm_T(\xb)$ such that $\Mb([-T,T];y) = \mbb_T >0 .$
  \begin{enumerate}
  \item  If $\al \in (1,2),$ then $y \in \Gm^*_T(\xb)$ is a classical solution of the Kepler-type problem with zero energy and $\Mb([-T,T];y) < \Mb([-T,T]; \xb);$ 
  \item If $\al =1$ and $|\phi_+ - \phi_-| < 2 \pi,$ then $ y \in \Gm^*_T(\xb)$ is a classical solution of the Kepler-type problem with zero energy and $\Mb([-T,T];y) < \Mb([-T,T]; \xb)$ and $|y(t)| >0$ for any $t \in [-T, T] \setminus \{0 \}.$
 \end{enumerate}

\end{thm}


First we introduce some notations that will be needed in the proof. 
\begin{dfi} \label{notations}
	For any  $\rob  > \rho \ge 0$ small enough and $T>0$, we set
	\begin{align*}
	\Gm_T(\xb; \ro) & := \{x \in \Gm_T(\xb): \min_{t \in [-T, T]} |x(t)| = \ro \}, \\
	\bar{m}_T(\ro) & := \inf\{ \Mb([-T,T];x): x \in \Gm_T(\xb; \ro) \},\\
	\mc{M}_T(\rho) & := \{ x \in \Gm_T(\xb): \Mb([-T,T];x) = \bar{m}_T(\rho) \}; \\
	\Gm_T(\xb; \ro, \bar{\ro}) & : = \{ x \in \Gm_T(\xb): \min_{t \in [-T, T]} |x(t)| \in [\ro, \rob] \};\\
	\mbb_T(\ro, \rob) &:= \inf\{ \Mb(x): x \in \Gm_T(\xb; \ro, \rob) \};\\
	\mc{M}_T(\ro, \rob) & : = \{ x \in \Gm_T(\xb; \ro, \rob): \Mb(x) = \mbb(\ro, \rob) \text{ and } \min_{t \in [-T, T]} |x(t)| < \rob \}. 
	\end{align*}
\end{dfi}

For the rest of this section, except in the proof of Theorem \ref{thm 3}, which will given at the very end, we will fix an arbitrary $T>0$ and assume
\begin{equation}
\label{nonzero}  0 < |\phi_+ - \phi_-| \le 2 \pi,
\end{equation}
To simplify notation, we omit the subindex $T$ from all the notions introduced in Definition \ref{notations} and the $[-T,T]$ from $\Mb([-T, T]; x).$

\begin{lm} \label{ro4}
 For any $\rho \ge 0$, $\mbb(\ro) > 0$ and $\mc{M}(\rho) \ne \emptyset. $
\end{lm}

\begin{proof}
	Once we showed $\mc{M}(\ro)$ is nonempty, it is obvious that $\mbb(\ro) >0$.
	
	It is well known that $\Mb$ is weakly lower semicontinuous with respect to the $H^1$ norm. By the standard argument in calculus of variations, it is enough to show that $\Mb$ is coercive on $\Gm(\xb;\ro).$
	
	For any $x \in \Gm(\xb; \ro)$, we define
	$$ \dl(x) := \max \{ |x(t_0) - x(t_1)|: t_0, t_1 \in [-T,T] \}.$$
	Let $\phi= \frac{\phi_- + \phi_+}{2}$, then for any $x \in \Gm(\xb; \ro)$ there is a $t_{\phi} \in (-T, T]$, such that 
	$$ \text{Arg}(x(t_{\phi})) = \phi \text{ or } |x(t_{\phi})| = 0. $$
	Simply using the law of sines and \eqref{nonzero}, we find there is a $\nu \in (0,2]$ independent of $x$, such that 
	$$ \langle x(0), x(t_{\phi}) \rangle \le (1 - \nu) |x(0)| |x(t_{\phi})|. $$
	This is the so called \emph{noncentral} condition introduced by Chen in \cite{Ch03}. In proposition $2$, \cite{Ch03}, he also showed that once the noncentral condition is satisfied, then there is a constant $C >0$, such that 
	\begin{equation}
	\label{noncentral} |x(t)| \le |x(-T)| + \dl(x) \le C \dl(x), \quad \forall x \in \Gm(\xb; \ro), t \in [-T, T]. 
	\end{equation}
	On the other hand by Cauchy-Schwarz inequality
	\begin{equation}
	\label{cs} \dl^2(x) \le ( \int_{-T}^T |\dot{x}(t)| \,dt )^2 \le 2T \int_{-T}^T |\dot{x}(t)|^2 \,dt. 
	\end{equation}
	Therefore 
	$$ \| x \|^2_{H^1} = \int_{-T}^T |x(t)|^2 \,dt + \int_{-T}^T |\dot{x}(t)|^2 \,dt \le C_1 \int_{-T}^{T} |\dot{x}(t)|^2 \,dt. $$
	By \eqref{noncentral},
	$$ \Vb(x(t)) = \frac{m_1}{\al |x(t)|^{\al}} \ge C_2 \dl^{-\al}(x), $$
	combining this with \eqref{cs}, we get 
	
	$$ \int_{-T}^{T} \Vb(x(t)) \,dt \ge C_3 (\int_{-T}^T |\dot{x}(t)|^2 \,dt)^{-\frac{\al}{2}}. $$
	Hence
	$$ \Mb(x) = \int_{-T}^T \ey |\dot{x}(t)|^2 \,dt \int_{-T}^T \Vb(x(t)) \,dt \ge C_4(\int_{-T}^T |\dot{x}(t)|^2 \,dt)^{\frac{2-\al}{2}} \ge C_5 \|x\|^{2 - \al}_{H^1}. $$
	Therefore when $\al \in [1, 2)$, $\Mb$ is coercive on $\Gm(\xb; \ro).$ 
\end{proof}

\begin{lm} \label{ro6}
	There is at least one $x \in \Gm(\xb; \ro, \rob)$ with $\Mb(x) = \mbb(\ro, \rob)>0 $.
\end{lm}

The proof of the above lemma is exactly the same as Lemma \ref{ro4}. We will not repeat it here. Notice that even with Lemma \ref{ro6}, $\mc{M}(\ro, \rob)$ may still be empty. 
	
\begin{lm} \label{ro5}
 $\xb|_{[-T,T]}$ is the only minimizer of $\Mb$ in $\Gm(\xb; 0)$, i.e.,
 $$ \mc{M}(0) = \{ \xb|_{[-T,T]} \}. $$
\end{lm}

This lemma follows from a straight forward computation and a similar argument used by Gordon in \cite{Go77}.

\begin{lm} \label{ro1}
 $$ \limsup_{\ro \to 0^+} \mbb(\ro) \le \mbb(0).$$
\end{lm}

\begin{proof}
	Without loss of generality let's assume $\phi_-=0$ and $\phi_+ = \phi.$ For any $\ro > 0$ small enough, we define 
	$$ \xr(t) = 
	   \begin{cases}
		\xb(t), & \text{ if } t \in \Om := [-T, T] \setminus [-\tr, \tr], \\
		\ro e^{i \om t}, & \text{ if } t \in [-\tr, \tr],
	\end{cases}
	$$
where $\tr = \mu^{-1} \ro^{\frac{2+\al}{2}}$ and $\om = \phi \mu \ro^{-\frac{2 + \al}{2}}. $ A straight forward computation shows that 

\begin{align*}
\Mb(\xr) - \Mb(\xb) & = \int_{\Om} \ey |\dot{\xb}|^2 \int_{-\tr}^{\tr} [\Vb(\xr) - \Vb(\xb)] + \int_{\Om} \Vb(\xb) \int_{-\tr}^{\tr} [\ey |\dot{x}_{\ro}|^2 -\ey |\dot{\xb}|^2]  \\
& \quad + \int_{-\tr}^{\tr} \ey |\dot{x}_{\ro}|^2 \int_{-\tr}^{\tr} \Vb(\xr) - \int_{-\tr}^{\tr} \ey |\dot{x}_{\ro}|^2 \int_{-\tr}^{\tr} \Vb(\xb) \\
& \sim O(\ro^{\frac{2-\al}{2}}). 
\end{align*}
This proves the lemma, because $\Mb(\xb) = \mbb(0)$ according to Lemma \ref{ro5}. 
\end{proof}

The above lemmas guarantees the existence of the obstacle minimizers. In the following we will get more information about these minimizers.

For any $\ro \ge 0$ and $x \in \Gmr$, we set
  $$T_{\ro}(x):= \{ t \in [-T, T]: |x(t)| = \ro \}.$$

\begin{lm}
  \label{restricted mini}
  If $\ro >0$ and $x \in \mc{M}(\ro),$ the following results hold.
  
  \begin{enumerate}
   \item For any $(a,b) \subset [-T, T] \setminus T_{\ro}(x)$, $x|_{(a,b)}$ is $C^2$ and a classic solution of the following equation
   \begin{equation} \label{mini1} \om^2 \ddot{x}(t) = \nabla \Vb(x(t)), \text{ where } \om^2 = \frac{\int_{-T}^{T} \Vb(x(t) \,dt}{\int_{-T}^{T} \frac{1}{2} |\dot{x}(t)|^2 \,dt}. \end{equation}
   \item There exist $ t^- \le t^+$, such that $T_{\ro}(x) = [t^-, t^+]$ and
   \begin{align*}
    |x(t)| > \ro, \quad  & \forall t \in [-T, t^-) \cup (t^+, T], \\
    |x(t)| = \ro, \quad & \forall t \in [t^-, t^+].
   \end{align*}
   \item The following energy identity holds for any $t \in [-T, T],$
   \begin{equation} \label{mini2} \frac{1}{2} |\dot{x}(t)|^2 - \frac{\Vb(x(t))}{\om^2} = 0. \end{equation}
   \item Set $x(t) = r(t) e^{i\tht(t)}$ in polar coordinates, then one of the following three situation must occur:
   \begin{enumerate}
    \item $t^- < t^+$ and $x \in C^1([-T, T], \cc)$;
    \item $t^- = t^+$ and $x \in C^1([-T, T], \cc)$;
    \item $t^- = t^+$ and $\dot{x}(t)$ doesn't exist at $t = t^-= t^+$.
   \end{enumerate}
   \item If $t^- < t^+$ then $\tht|_{(t^-, t^+)}$ is $C^2$, strictly monotone and a classic solution of the following equation
   $$ \ddot{\tht}(t) = \frac{1}{\om^2} \cdot \frac{1}{\ro} \langle \nabla \Vb(\ro e^{i \tht(t)}), ie^{i \tht(t)} \rangle. $$
   \item $\dot{r}(t)> 0$, if $t \in (t^+, T]$ and $\dot{r}(t)< 0.$ if $ t \in [-T, t^-).$
    \end{enumerate}
\end{lm}

\begin{proof}
 The proof of ($4$) can be found in the proof of Proposition $2.6$ in \cite{BTV14}. The proofs of the rest results can be found in the proof of Lemma $4.30$ in \cite{ST12}.
 
 The last result, $(6)$ is easy to see, once the reader notice that once other results are true. Then $x(t), t \in [-T, t^-) \cup (t^+, T]$ is a parabolic solution (zero energy) of the following Kepler-type problem
 $$ \ddot{x}(t) = - \frac{m_1 x(t)}{\om^2 |x(t)|^{\al +2} } $$
 with $r(t^{\pm}) = \min \{|x(t)|: t \in [-T, T] \}. $

\end{proof}

In our proof we need the obstacle minimizers to have enough regularity, i.e., at least $C^1$. However the above lemma does not guarantees that. On the other hand a minimizer contained in $\mc{M}(\ro, \rob)$ is $C^1.$

\begin{lm}
	\label{ro2}  If $x \in \mc{M}(\ro, \rob)$, then $x \in C^1([-T, T], \cc)$.
\end{lm}

\begin{proof}
	First if $x$ satisfies 
	$$ \min \{ |x(t)| : \text{ for any } t \in [-T,T] \} > \ro,$$
	then $x$ is in the interior of the set $\Gm(\xb; \ro, \rob)$ and it is a classical solution of the Kepler-type problem, so $x(t)$ must be $C^1$. 
	
	If  
	$$ \min \{ |x(t)| : \text{ for any } t \in [-T,T] \} = \ro,$$
	then $x \in \mc{M}(\ro)$ and it must satisfies all the properties in Lemma \ref{restricted mini}. Therefore we only need to show the case $(c)$ in Lemma \ref{restricted mini} $(4)$ will not happen. 
	
	Assume this is the case, let $t_0 \in (-T, T)$ be the only moment that $\dot{x}$ does not exist, by a small perturbation of $x$ near $x(t_0)$ in the direction away from the origin we can get a new with strictly smaller Maupertuis' functional and still contained in $\Gm(\ro, \rob).$ A detailed argument can be found in the proof of Lemma $4.30$ \cite{ST12}.
	
\end{proof}

\begin{lm} \label{ro3} For any $0 < |\phi_+ - \phi_-| \le 2 \pi$,
\begin{enumerate}
 \item if $\al \in (1, 2)$, then there is a $\ro^* > 0$, small enough, such that $\mc{M}(\ro, \rob) = \emptyset,$ for any $0< \ro < \rob \le \ro^*;$
 \item if $\al = 1$ and $|\phi_+- \phi_-| < 2 \pi$, then there is a $\ro^* > 0$, small enough, such that $\mc{M}(\ro, \rob) = \emptyset$ for any $0< \ro < \rob \le \ro^*.$
\end{enumerate}

\end{lm}

Let's postpone the proof of Lemma \ref{ro3} for a moment and see how we can use it to proof Theorem \ref{collisionless}.

\begin{proof}

[Theorem \ref{collisionless}]
 We will give the detailed proof for the case $\al \in (1,2)$, while the other is similar.
 
 First by the same argument as in the proof of Lemma \ref{ro4}, it is easy to see there is a $y \in \Gm(\xb)$ with $\Mb(y) = \mbb$ and $\mbb > 0$. 
 
 Assume there is a $t_0 \in [-T, T]$ with $y(t_0) = 0$, then $ \Mb(y) = \mbb(0)$ and $y \in \mc{M}(0)= \mbb$.
 
 For any $0 < \ro < \rob < \ro^*$, by Lemma \ref{ro6}, there is a $x_{\rob} \in \Gm(\xb; \ro, \rob)$ with $\Mb(x_{\rob}) = \mbb(\ro, \rob)$. However by Lemma \ref{ro3}, $x_{\rob} \notin \mc{M}(\ro, \rob).$ This means
 $$ \min \{ |x(t)|: t \in [-T, T] \} = \rob,$$
 so $ x_{\rob} \in \Gm(\xb; \rob)$ and $\Mb(x_{\rob}) = \mbb(\rob).$
 
 Meanwhile by Lemma \ref{ro4}, there is a $x_{\ro} \in \Gm(\xb; \ro)$ with $\Mb(x_{\ro}) = \mbb(\ro).$ Obviously $x_{\ro} \in \Gm(\xb; \ro, \rob)$ and by Lemma \ref{ro2}
 $$ \Mb(x_{\ro}) = \mbb(\ro) > \mbb(\ro, \rob) = \mbb(\rob) = \Mb(x_{\rob}).$$
 The above argument implies $\mbb(\ro)$ is strictly decreasing with respect to $\ro \in (0, \ro^*].$
 
 Combining this with Lemma \ref{ro1}, we have
 $$ \lim_{t \to 0^+} \mbb(\ro) = \limsup_{t \to 0^+} \mbb(\ro) \le \mbb(0) = \mbb.$$
 Hence 
 $$ \mbb(\ro^*) < \mbb,$$
 which is absurd.

\end{proof}

Now we are ready to prove Lemma \ref{ro3}.
\begin{proof}

 [Lemma \ref{ro3}]
 An contradiction argument will be used here as well. Suppose our results are not true, then there are two sequences $\{ \ron \}$ and $\{\robn\}$ with $0 < \ron < \robn$, both converge to $0$ as $n$ goes to positive infinity. Furthermore for each $n,$ there is $x_n \in \Gm(\xb; \ron),$ satisfying
 $$ \Mb(x_n) = \mbb_{\ron} = \mbb_{\ron, \robn} . $$

 Obviously each $x_n$ satisfies Lemma \ref{restricted mini}, so there are $t_n^- \le t_n^+$ such that
 $$ T_{\ron}(x_n) := \{ t \in [-T, T]: x_n(t) = \ron \} = [t_n^-, t_n^+].$$

 Even though $x_n$ is a minimizer in $\Gm(\xb; \ron)$, by Lemma \ref{restricted mini}, it is possible that $x_n$ is not $C^1$ on $[-T, T]$. However by exact the same arguments as in the proof of part $(vi),$ Lemma $4.30$ in \cite{ST12}, we know $x_n|_{[-T,T]}$ must be $C^1$.

 Write each $x_n(t)$ in polar coordinates: $x_n(t) = r_n(t) e^{i \ttn(t)}$.

 First let us focus on the time interval $ T_{\ron}(x_n)=[\tna, \tnb]$, then
 $$ x_n(t) = \ron e^{i \ttn(t)},  \quad \forall t \in [\tna, \tnb].$$
 Hence by the energy identity, for any $t \in (\tna, \tnb),$
 $$ \frac{1}{2} |\dot{x}_n|^2 = \frac{1}{2} |\ron i \dot{\tht}_n e^{i \ttn}|^2 = \frac{1}{2} \ron^2 \dot{\tht}_n^2 = \frac{\Vb(x_n)}{\om_n^2} = \frac{1}{\om_n^2} \cdot \frac{m_1}{\al |\ron|^{\al}}. $$
 Therefore
 $$ \ron^2 \dot{\tht}_n^2 = \frac{2 m_1}{\om_n^2 \al} \ron^{-\al}. $$

 By Lemma \ref{restricted mini}, $\dot{\tht}_n$ is positive for $t \in (\tna, \tnb),$  so
 $$ \dot{\tht}_n(t) = \sqrt{\frac{2m_1}{\om_n^2 \al} } \ron^{-\frac{2+\al}{2}}, \quad \forall t \in (\tna, \tnb).$$
 Therefore
 \begin{equation}
  \label{am2}  \ttn(\tnb) - \ttn(\tna) = \int_{\tna}^{\tnb} \dot{\tht}_n(t) \,dt= \sqrt{\frac{2m_1}{\om_n^2 \al} } \ron^{-\frac{2+\al}{2}} (\tnb - \tna).
 \end{equation}

 Meanwhile for $t \in (\tna, \tnb)$, the angular momentum
\begin{equation} \label{am1}
J(x_n(t)) = \ron^2 \dot{\tht}_n(t) = \sqrt{\frac{2m_1}{\om_n^2 \al} } \ron^{\frac{2-\al}{2}}.
\end{equation}

 On the other hand, if $t \in (-T, \tna)$ (or $ t \in (\tnb, T)$), $x_n(t)$ is a solution of the Kepler-type  equation \eqref{mini1}, so the angular momentum $J(x_n)$ is a constant, i.e.,
 \begin{equation}
  \label{am3} J(x_n(t)) = x_n(t) \times \dot{x}_n(t) = constant, \quad \forall t \in (-T, \tna) (\text{ or }  t \in (\tnb, T) ).
 \end{equation}
 Then we have the following result.

 \begin{lm}
  \label{am4}
  $$ J(x_n(t)) \equiv \frac{2m_1}{\om_n^2 \al} \ron^{\frac{2 - \al}{2}}, \quad \forall t \in (T, T). $$
 \end{lm}
 \begin{proof}

 [Lemma \ref{am4}]
 By \eqref{am1} and \eqref{am3}, there are two constants $c_1, c_2$ such that
 $$ J(x_n(t)) =
 \begin{cases}
  c_1,  & \text{ when } t \in (-T, \tna)\\
  \frac{2m_1}{\om_n^2 \al} \ron^{\frac{2 - \al}{2}}, & \text{ when } t \in (\tna, \tnb) \\
  c_2, & \text{ when } t \in (\tnb, T)
 \end{cases}
 $$
 By Lemma \ref{restricted mini}, $x_n \in C^1((-T, T), \cc)$, therefore $J(x_n) = x_n \times \dot{x}_n \in C^0((-T, T), \cc)$, which means
 $$ J(x_n(\tna)) = \lim_{t \to (\tna)^-} J(x_n(t)) = c_1 = \lim_{ t \to (\tna)^+} J(x_n(t)) = \frac{2m_1}{\om_n^2 \al} \ron^{\frac{2 - \al}{2}}.$$
 Similarly
 $$ J(x_n(\tnb)) = c_2 = \frac{2m_1}{\om_n^2 \al} \ron^{\frac{2 - \al}{2}}. $$
 \end{proof}

 With Lemma \ref{am4}, now we can compute the change of the angular variable $\ttn(t)$ when $t$ goes from $t_n^+$ to $T.$
 \begin{equation}
  \label{ac1} \ttn(T) - \ttn(\tnb) = \int_{\tnb}^{T} \dot{\tht}_n(t) \,dt = \int_{\tnb}^{T} \frac{J(x_n)}{r_n^2(t)} \, dt .
 \end{equation}
 Since $x_n$ satisfies the energy identity \eqref{mini2},
 $$ \frac{1}{2} |\dot{x}_n|^2 = \frac{1}{\om_n^2} \cdot \frac{m_1}{\al |x_n|^{\al}}.$$
 Rewrite the above equation using polar coordinates
 $$\frac{1}{2} |\dot{r}_n e^{i \ttn} + ir_n \dot{\tht}_n e^{i \ttn}|^2 = \frac{1}{2}(\dot{r}_n^2 + r_n^2 \dot{\tht}_n^2) = \frac{1}{\om_n^2} \cdot \frac{m_1}{\al r_n^{\al}}, $$

 $$ \dot{r}_n^2 = \frac{2m_1}{\om^2_n \al} \cdot \frac{1}{r^{\al}_n} -r_n^2 \dot{\tht}_n^2 = \frac{2m_1}{\om^2_n \al} \cdot \frac{1}{|r_n|^{\al}} - \frac{J^2(x_n)}{r_n^2}.$$
 By $(6)$, Lemma \ref{restricted mini}, $\dot{r}_n(t) >0$, for $ t \in (t_n^+, T].$
 Therefore
 \begin{equation}
  \label{ac2} \frac{dr_n}{dt} = \dot{r}_n = \sqrt{ \frac{2m_1}{\om^2_n \al} \cdot \frac{1}{r_n^{\al}} - \frac{J^2(x_n)}{r_n^2}}.
 \end{equation}
 The negative square root was dropped because of the following lemma.

 By \eqref{ac2}, we can change the integral variable in \eqref{ac1} from $t$ to $r_n$
 $$ \ttn(T) - \ttn(\tnb) = \int_{\ron}^{r_n(T)} \frac{J(x_n)}{r_n^2} \frac{1}{\sqrt{ \frac{2m_1}{\om^2_n \al} \cdot \frac{1}{r_n^{\al}} - \frac{J^2(x_n)}{r_n^2}}} \, dr_n $$

 By Lemma \ref{am4}, $J(x_n) = \frac{2m_1}{\om_n^2 \al} \ron^{\frac{2 - \al}{2}}$ plug it into the above equation we get
 \begin{equation} \label{ac3}
  \ttn(T) - \ttn(\tnb) = \int_{\ron}^{r_n(T)} \frac{1}{\sqrt{\ron^{\al-2} r_n^{4-\al} - r_n^2}} \,dr_n = \int_{\ron}^{r_n(T)} \frac{1}{r_n \sqrt{(\frac{\ron}{r_n})^{\al-2} -1}} \,dr_n.
 \end{equation}

 Notice that $r_n(T) = |x_n(T)| \equiv |\bar{x}(T)|$ for any $n$, we set $r^*:= |\bar{x}(T)|$. Set $\xi_n = \frac{\ron}{r_n}$, then $\,dr_n = -\xi_n^{-2} \ron \,d\xi_n$. We can simplify \eqref{ac3} to
 \begin{equation}
  \label{ac4} \ttn(T) - \ttn(\tnb) = - \int_{1}^{\frac{\ron}{r^*}} \frac{1}{\sqrt{\xi_n^{\al} - \xi_n^2}} \,d\xi_n = \int_{\frac{\ron}{r^*}}^{1} \frac{1}{\sqrt{\xi_n^{\al} - \xi_n^2}} \,d\xi_n.
 \end{equation}
 Furthermore if we set $\xi_n = \eta_n^{\frac{2}{2-\al}}$, then $ \,d\xi_n = \frac{2}{2-\al} \eta_n^{\frac{\al}{2 - \al}} \,d\eta_n.$ Plug this into \eqref{ac4}
 \begin{equation}
  \label{ac5} \ttn(T) - \ttn(\tnb) = \frac{2}{2 -\al} \int_{\frac{\ron}{r^*}}^{1} \frac{1}{\sqrt{1 - \eta_n^2}}\,d\eta_n = \frac{2}{2-\al}(\frac{\pi}{2} - \sin^{-1}(\frac{\ron}{r^*})).
 \end{equation}
 Similarly
 \begin{equation}
  \ttn(\tna) - \ttn(-T) = \frac{2}{2-\al}(\frac{\pi}{2} - \sin^{-1}(\frac{\ron}{r^*})).
 \end{equation}
 By Lemma \ref{restricted mini} says $\ttn(t)$ is monotone increasing when $ t \in [\tna, \tnb]$, therefore
 $$ \ttn(T) - \ttn(-T) \ge \ttn(T) - \ttn(\tnb) +  \ttn(\tna) - \ttn(-T) = \frac{2}{2-\al}(\pi - 2 \sin^{-1}(\frac{\ron}{r^*})). $$
 Recall that $\ron \to 0$, as $n \to +\infty$, so $$\lim_{n \to +\infty} \frac{\ron}{r^*} = 0.$$
 Therefore if $\al \in (1, 2)$, we have
 $$ \lim_{n \to +\infty} \ttn(T) - \ttn(-T) \ge \lim_{n \to +\infty}  \frac{2}{2-\al}(\pi - 2 \sin^{-1}(\frac{\ron}{r^*})) = \frac{2}{2-\al}\pi > 2 \pi. $$
 Which is absurd, because for any $n \in \mb{N},$
 $$\ttn(T) - \ttn(-T) = \phi_+ - \phi_- \le 2 \pi.$$
 At the same time when $\al =1$ and $\phi_+ - \phi_- < 2 \pi$, we get
 $$ \phi_+ - \phi_- = \lim_{n \to +\infty} \ttn(T) - \ttn(-T) \ge \lim_{n \to +\infty}  2(\pi - 2 \sin^{-1}(\frac{\ron}{r^*})) = 2\pi.$$
 which is again a contradiction.
 This finishes our proof of Lemma \ref{ro3}.

\end{proof}

Now we are ready to prove Theorem \ref{thm 4} and Theorem \ref{thm 3}. 

\begin{proof}

 [Theorem \ref{thm 4}]
 According to  Lemma $3.1$, \cite{BTV14}, which says that there is a one-to-one correspondence between minimizers of $\ab$ in $ \cup_{S > 0} \Gm_S(\xb; T)$ and minimizers of $\Mb$ in $\Gm_T(\xb),$ Theorem \ref{thm 4} is a direct corollary of Theorem \ref{collisionless}.
\end{proof}

Now let's see how we can use Theorem \ref{thm 4} to prove Theorem \ref{thm 3}. 

\begin{proof}

 [Theorem \ref{thm 3}]
 We will only give the details for case $\al \in (1,2)$, while the other is similar.
 
 The result is trivial, when $\phi_- = \phi_+$, so we will assume $0 < |\phi_+ - \phi_-| \le 2 \pi.$ Using the noncentral condition and same argument as in the proof of Lemma \ref{ro4}, it is easy to see $\bar{A}_T$ is coercive in $\Gm_T(\xb)$ and by the standard argument of calculus of variations, there is at least one $\gm \in \Gm_T(\xb)$ with $\bar{A}_T(\gm) = \bar{c}_T(\xb). $
 
 All we need to show is that $\gm$ is free of collision. By a classical result of Gordon \cite{Go77}, if $\gm$ contains a collision, then $\gm = \xb$, because $\gm$ is a minimizer. 
 
 Therefore we are done, once we can find a collisionless curve $y \in \Gm^*_T(\xb)$ with 
 $$ \bar{A}_T(y) < \bar{A}_T(\xb). $$

 By Theorem \ref{thm 4}, there is a $S> 0$ and $x \in \Gm_S(\xb; T)$, such that $\ab_S(x) < \ab_T(\xb)$ and $|x(t)| \ne 0$, for any $t \in [-S, S].$

 Let $\ep = \ab_T(\xb) - \ab_S(x) > 0$. We choose a $T_1 > \max \{ S, T\}$ and define a new curve
 $$ z(t) =
  \begin{cases}
   \xb(at+b), & \text{ if } t \in [S, T_1] \\
   x(t), & \text{ if } t \in [-S, S] \\
   \xb(at -b), & \text{ if } t \in [-T_1, -S],
  \end{cases}
  $$
  where $a = \frac{T_1-T}{T_1 -S}, b = T_1 \frac{T-S}{T_1 -S}. $
  Then
  \begin{align*}
   \ab_{T_1}(\xb) - \ab_{T_1}(z) & = \ab([T, T_1], \xb) - \ab([S, T_1]; z) + \ab([-T_1, T]; \xb) - \ab([-T_1, S]; z) \\
                                 & + \ab_T(\xb) - \ab_S(x). \\
  \end{align*}
  By the definition of $\xb$, it is obvious that
  $$ \ab([T, T_1], \xb) - \ab([S, T_1]; z) = \ab([-T_1, T]; \xb) - \ab([-T_1, S]; z). $$
  If we set
  $$ f(T_1) := \ab([T, T_1], \xb) - \ab([S, T_1]; z), $$
  then
  \begin{equation}
   \label{f ep} \ab_{T_1}(\xb) - \ab_{T_1}(z) = 2f(T_1) + \ep.
  \end{equation}

 A simple calculation shows

 \begin{align*}
  |f(T_1)| &= |(1-a) \int_T^{T_1} \ey |\dot{\xb}(t)|^2 \,dt + (1-\frac{1}{a}) \int_T^{T_1} \frac{m_1}{\al |\xb(t)|^{\al}} \,dt| \\
  & \le |1-a| \frac{2}{(2+\al)^2} \mu^{\frac{4}{2 + \al}} \int_T^{T_1} t^{-\ale} \,dt + |1-\frac{1}{a}| \frac{m_1}{\al}  \mu^{-\ale} \int_T^{T_1} t^{- \ale} \,dt \\
  & = ( C_1 \frac{|T-S|}{T_1 -S} + C_2 \frac{|S-T|}{T_1 -T}) ( T_1^{\frac{2 -\al}{2 + \al}} - T^{\frac{2-\al} {2 + \al}})
  \end{align*}
  Since $\al \in [1, 2),$ $0< \frac{2-\al}{2 + \al} <1$. Hence it is not hard to see
  $$ |f(T_1)| \to 0, \text{ as } T_1 \to 0. $$
  Therefore for $T_1$ large enough
  $$ \ab_{T_1}(\xb) - \ab_{T_1}(z) = 2f(T_1) + \ep >0. $$

  Set $\lmd = \frac{T_1}{T}$ and define
  $$ z_{\lmd}(t) = \la z(\lmd t), \quad t \in [-\frac{T_1}{\lmd}, \frac{T_1}{\lmd}]; $$
  $$ \xb_{\lmd}(t) = \la \xb(\lmd t), \quad t \in [-\frac{T_1}{\lmd}, \frac{T_1}{\lmd}].$$

  It is easy to see $\xb_{\lmd}(t) = \xb(t)$ and $z_{\lmd} \in \Gm^*_T(\xb).$ As in the proof of Lemma \ref{lmd 1}, a straight forward calculation shows
  $$ \ab_T(\xb) - \ab_T(z_{\lmd}) = \la [ \ab_{T_1}(\xb) - \ab_{T_1}( z_{\lmd})] = \la(2 f(T_1)+ \ep) > 0.$$
  By the definition of $z_{\lmd}$, $|z_{\lmd}(t)| \ne 0$, for any $t \in [-T, T]$ and we are done.
\end{proof}

\emph{Acknowledgements.} The author wish to express his gratitude to Prof. Ke Zhang for his continued support and encouragement.

\bibliographystyle{habbrv}
\bibliography{ref-time}

\end{document}